\documentclass[11pt]{article}

\usepackage{amssymb}
\usepackage{amsthm}
\usepackage{amsmath}
\usepackage{graphicx} 
\usepackage[usenames,dvipsnames]{xcolor}
\usepackage[margin = 1in]{geometry}
\usepackage{enumerate}
\usepackage[toc,page]{appendix}
\usepackage{lineno}
\usepackage{color}
\usepackage{mhchem}
\usepackage{siunitx}

\usepackage{subfigure}
 \usepackage{tikz}
\usepackage{multirow}

\definecolor{mygreen}{rgb}{0.1,0.75,0.2}

\providecommand{\bbs}[1]{\left(#1\right)}

 \newtheorem{thm}{Theorem}[section]
 
 \newtheorem{lem}[thm]{Lemma}
 \newtheorem{prop}[thm]{Proposition}

 \theoremstyle{definition}
 \newtheorem{defn}{Definition}

 \theoremstyle{remark}
 \newtheorem{rem}{Remark}

 \numberwithin{equation}{section}

\newcommand{\la}{\langle}
\newcommand{\ra}{\rangle}
\newcommand{\pt}{\partial}

\newcommand{\ud}{\,\mathrm{d}}

\newcommand{\sL}{\mathcal{L}}

\newcommand{\bR}{\mathbb{R}}

\newcommand{\invm}{\mathcal{M}}                                  
\DeclareMathAlphabet{\mathbfsf}{\encodingdefault}{\sfdefault}{bx}{n}
\newcommand*{\vecc}[1]{\mathbfsf{#1}}                            

\usepackage[xcolor=dvipdf, authormarkup=none]{changes}

\usepackage[style=numeric-comp,%
            abbreviate=true,   
            useprefix=true,    %
            giveninits=false,  %
            hyperref=false,    %
            maxcitenames=2,    %
            maxbibnames=20,    %
            uniquename=init,   %
            sortcites=true,    
            doi=true,          %
            isbn=true,         %
            url=false,         %
            eprint=false,      %
            backend=bibtex     %
           ]{biblatex}         
  \renewbibmacro{in:}{}
  \DefineBibliographyStrings{english}{bibliography = {References}}

\begin{document}

\title{Structure preserving schemes for  Fokker-Planck equations of irreversible processes}

\author{Chen Liu\footnote{Department of Mathematics, Purdue University, West Lafayette, IN 47907, USA (liu3373@purdue.edu).} \  \ 
Yuan Gao\footnote{Department of Mathematics, Purdue University, West Lafayette, IN 47907, USA (gao662@purdue.edu). YG's research was supported by NSF grant DMS-2204288.} \  \  
	and \ Xiangxiong Zhang\footnote{Department of Mathematics, Purdue University, West Lafayette, IN 47907, USA (zhan1966@purdue.edu). XZ's research was supported by NSF grant DMS-1913120.} }

\maketitle

\begin{abstract}
In this paper, we construct structure preserving schemes for solving Fokker-Planck equations associated with irreversible processes. The proposed method is first order in time. We consider two structure-preserving spatial discretizations, which are second order and fourth order accurate finite difference schemes. They are derived via finite difference implementation of the classical  $Q^k$ ($k=1,2$) finite element methods on uniform meshes. Under mild mesh conditions and practical time step constraints, the schemes are proved monotone, thus are positivity-preserving and energy dissipative. In particular, our scheme is suitable for capturing steady state solutions in large final time simulations.
\\ \\
\textbf{Key words}. Fokker-Planck equation, finite difference, monotonicity, positivity, energy dissipation, high-order accuracy
\\
\textbf{AMS subject classifications}. 65M06, 65M12, 65M60
\end{abstract}

\section{Introduction}
Irreversible drift-diffusion processes  are a class of important stochastic processes in physics and chemistry. For instance, an irreversible drift-diffusion process can model non-equilibrium biochemical reactions, which possess non-equilibrium steady states (NESS). The most important features for non-equilibrium reactions are nonzero fluxes and positive entropy production rate at NESS.  These features in an irreversible biochemical reaction maintain a circulation at NESS and we refer to pioneering studies by \textsc{Prigogine} \cite{prigogine1968introduction}. 
Let $\vec{b}: \bR^d \to \bR^d$ be a general drift field depending only on the state variable ${\vec{x}}\in \bR^d$. 
Consider a stationary   drift-diffusion process with  white noise $B_t$ that satisfies a stochastic differential equation (SDE) for ${\vec{x}}_t\in \bR^d$
\begin{equation}\label{sde-x}
\ud {\vec{x}}_t = \vec{b}({\vec{x}}_t) \ud t + \sqrt{2}\sigma  \ud B_t.
\end{equation}
In general, $\sigma$ is a   noise matrix and 
 $D:=\sigma \sigma^T\in \bR^{d \times d}$. 
For simplicity, in this paper we only discuss the simple case where $D>0$ is a constant number.  
By Ito's formula, the corresponding Fokker-Planck equation  for SDE \eqref{sde-x}, i.e., the Kolmogorov forward equation   for    density $\rho({\vec{x}}, t)$, is given by 
\begin{equation}\label{FP-N}
\pt_t \rho = \sL^* \rho:=-\nabla\cdot(\vec{b} \rho)+ \nabla \cdot  ( D\nabla \rho ),
\end{equation}
In terms of \eqref{FP-N}, the irreversibility means that one cannot find an invariant measure $\pi$ such that the generator $\sL$ is symmetric in $L^2(\pi)$. Irreversibility has many equivalent characterizations. One is equivalent to that it is impossible to write the drift $\vec b$ in a potential form $\vec{b}= -D \nabla \varphi$ for any potential $\varphi$. See another equivalent irreversibility condition \eqref{Fs}.
Irreversible processes and the associated Fokker-Planck equations can be used to describe more general dynamic processes, such as the biochemical reactions with non-equilibrium steady state \cite{hill2005free, QianBook, GL22}, they can also be used as sampling acceleration and variance reduction \cite{turitsyn2011irreversible, duncan2016variance}. 
\par
In this paper, we focus on studying second and fourth order in space numerical schemes for \eqref{FP-N} with a general drift field $\vec{b}$. We will prove the proposed high-order schemes preserve (\emph{i}) the conservation of total mass, (\emph{ii}) the positivity of $\rho$, (\emph{iii}) the energy dissipation law with respect to $\phi$-entropy, and (\emph{iv}) the exponential convergence to equilibrium $\invm$. 
To be more precise, we consider the equation \eqref{FP-N} in a bounded domain $\Omega\subset \bR^d$ with no-flux boundary condition
\begin{equation}\label{BC}
-\rho\vec{b} \cdot \vec{n} + D \nabla \rho \cdot \vec{n} = 0 \quad \text{ on }\Gamma, 
\end{equation}
where $\vec{n}$ is the unit outer normal of the boundary $\Gamma=\pt \Omega$.
Let $\rho^0({\vec{x}})$ be the initial condition to \eqref{FP-N}.
Under the no-flux boundary condition, it is easy to verify the conservation of total mass
\begin{equation}\label{mass}
\int_\Omega \rho \ud { x} 
= \int_\Omega \rho^0 \ud { x}.
\end{equation} 
\par
Designing structure preserving high order numerical schemes for the irreversible Fokker-Planck equation \eqref{FP-N} is very important, not only because the irreversible processes are able to describe lots of fundamental non-equilibrium behaviors, such as circulations at NESS in an ecosystem, but also because of a general drift field is commonly used to construct acceleration or control for a given stochastic process or for a process constructed from discretization of irreversible Fokker-Planck equations; see for instance the optimally controlled transition path computations \cite{weinan2006, zhou2016, gao2022data,  gao2020transition, wei2022optimal} and the accelerated sampling and optimization \cite{turitsyn2011irreversible, duncan2016variance, ZWE, yg20}. 

\subsection{Invariant measure $\invm$ and $\invm$-symmetric decomposition}
Assume that there exists a positive invariant  measure $\invm\in \mathcal{C}^1(\bar{\Omega})$, $\int \invm \ud x =1$ and $\invm {\geq\epsilon_0}>0$. $\invm$ satisfies the static equation
\begin{equation}\label{pi}
\nabla \cdot F^s := \nabla \cdot \bbs{ \vec{b} \invm- D \nabla\invm }  = 0,
\end{equation}
and the same no-flux boundary condition 
\begin{equation}\label{BCm}
-\invm\vec{b} \cdot \vec{n}  + D \nabla \invm \cdot \vec{n} = 0 \quad \text{ on }\Gamma.
\end{equation}
For our compact domain $\bar{\Omega}$, such an invariant measure always exists, cf.  \cite{cattiaux1992stochastic}.
The irreversibility is then equivalently characterized as that the steady flux is not pointwisely zero
 \begin{equation}\label{Fs} 
F^s =   \vec{b} \invm- D \nabla\invm \neq 0.
\end{equation}
The special invariant measure such that $F^s=0$ is also called a detailed balanced invariant measure.

\par
Leveraging the existence of the positive invariant measure $\invm$, one can utilize $\invm$ to decompose the irreversible Fokker-Planck equation \eqref{FP-N} into a dissipative part and a conservative part. The dissipative part is a gradient flow, which represents the decay from any initial density to the invariant measure. Meanwhile, the conservative part preserves the total energy and maintains a nonzero equilibrium flux \cite{hill2005free, QianBook}.  
To this end, let us derive the $\invm$-symmetric decomposition and the associated energy dissipation relation.

\par
With a positive invariant measure $\invm$, we decompose \eqref{FP-N} into the sum of a gradient flow part and a Hamiltonian flow part
\begin{equation}\label{pi-decom}
\begin{aligned}
\pt_t \rho &= \nabla \cdot \bbs{D \nabla \rho - \vec{b} \rho}
= \nabla \cdot \bbs{ D \invm \nabla \frac{\rho}{\invm} +  \frac{\rho}{\invm} \bbs{D\nabla \invm -   \invm \vec{b}} }\\
 &=  \nabla \cdot \bbs{ D \invm \nabla \frac{\rho}{\invm}} +     \bbs{D\nabla \invm -   \invm \vec{b}} \cdot \nabla \frac{\rho}{\invm}.
\end{aligned}
\end{equation}
From \eqref{pi}, we know 
\begin{equation}
\vec{u}:= D\nabla \invm - \invm \vec{b}, \quad \nabla \cdot \vec{u}=0 \quad \text{ in }\Omega.
\end{equation}
The reversibility condition for the drift-diffusion process becomes $\vec{u}\equiv 0$, but we focus on more general case that $\vec{u}\neq 0$. Using the notation $\vec{u}= D\nabla \invm - \invm \vec{b}$ and \eqref{BCm}, we know 
\begin{equation}\label{ttt}
\vec{u}\cdot \vec{n} = 0 \quad \text{ on } \Gamma.
\end{equation}
 By the exactly same decomposition in \eqref{pi-decom}, the  no-flux boundary condition \eqref{BC}
becomes
 \begin{equation}\label{bcbc}
 \bbs{D \invm \nabla \frac{\rho}{\invm} +  \frac{\rho}{\invm} \bbs{D\nabla \invm -   \invm \vec{b}}} \cdot \vec{n} =0 \quad \text{ on } \Gamma.
 \end{equation}
 This, together with \eqref{ttt}, implies $D\invm  \nabla \frac{\rho}{\invm} \cdot \vec{n} = 0$ on $\Gamma$.
Thus, we conclude that 
\begin{equation}
D\invm  \nabla \frac{\rho}{\invm} \cdot \vec{n} = 0, \qquad \vec{u}\cdot \vec{n} = 0 \quad \text{ on } \Gamma.
\end{equation}
\par
For certain applications, in the case of the invariant measure $\invm>0$ is given and satisfies \eqref{pi} and \eqref{BCm}, we can utilize it to construct a numerical scheme for the following equation, which stems from above decomposition, in conservative form:
\begin{equation}\label{model1}
\begin{aligned}
\partial_t\rho = \nabla\cdot\left(D\invm\nabla{\frac{\rho}{\invm}}\right) + \nabla\cdot{\left(\vec{u}\,\frac{\rho}{\invm}\right)}& && \text{in}~\{t>0\}\times\Omega,\\
\rho = \rho^0& && \text{in}~\{t=0\}\times\Omega,\\
D\invm \left(\nabla{\frac{\rho}{\invm}}\right)\cdot\vec{n} = 0, \quad \vec{u}\cdot \vec{n}=0& && \text{on}~\{t>0\}\times\partial{\Omega}.
\end{aligned}
\end{equation}
Let us refer to \eqref{model1} as {  \textbf{Model~1}.} 
In this model, we highlight the positive invariant measure $\invm$ is prescribed and the vector field $\vec{u}$ is a given time-independent continuously differentiable function, which satisfies $\nabla \cdot \vec{u}=0$. In general, this $\vec{u}$ could be prescribed directly, or computed from the original drift $\vec{b}$ in \eqref{FP-N}. 
\par
Let us show the energy dissipation relation of \textbf{Model~1.}
Define operators $L^\ast$ and $T$ as follows
\begin{align}\label{defLT}
{L^\ast}:= \nabla \cdot\bbs{D \invm \nabla}, \quad
{T}:= (D\nabla \invm - \invm \vec{b}) \cdot \nabla=\vec{u}\cdot \nabla.
\end{align}
Using the boundary condition in \eqref{model1}, it is easy to check that $L^*$ is a symmetric and nonnegative operator in $L^2(\Omega)$, namely
\begin{equation}
(f, L^\ast g) = (L^\ast f, g), \quad 
(f, L^\ast f) \leq 0, \quad 
\forall f,g\in L^2(\Omega) \,\,\mbox{satisfying}\,\, D\invm \nabla f \cdot \vec{n}\big|_{\Gamma}=0. 
\end{equation}
Here, the notation $(\cdot, \cdot)$ denotes the standard $L^2$-inner product.
From the properties of $\vec{u}$, it is straightforward to see the operator $T$ is asymmetric in $L^2(\Omega)$, namely
\begin{equation}\label{antiT}
(f, T g) = -(T f, g), \quad  (f, T f) =0, \quad \forall  f,g\in L^2(\Omega).
\end{equation}
Therefore, define free energy $E := \int \phi\bbs{\frac{\rho}{\invm}} \invm \ud x$, for any convex function $\phi$, we have the following energy dissipation relation. 
\begin{equation}\label{energy1}
\begin{aligned}
\frac{\ud E}{\ud t} &= (\phi'\bbs{\frac{\rho}{\invm}}, \pt_t \rho)
= (\phi'\bbs{\frac{\rho}{\invm}}, ({ L^*} +{ T}) \frac{\rho}{\invm})\\
&= -\int \phi''\bbs{\frac{\rho}{\invm}}\invm \nabla \frac{\rho}{\invm}\cdot D \nabla \frac{\rho}{\invm} \ud {x}
\leq 0.
\end{aligned}
\end{equation}
Notice, in above, we used the identity
\begin{equation*}
(\phi'\bbs{\frac{\rho}{\invm}}, \vec{u}\cdot \nabla \frac{\rho}{\invm}) = (\vec{u}, \nabla \phi\bbs{\frac{\rho}{\invm}}) =0,
\end{equation*}
which is due to the integration by parts and properties $\nabla \cdot \vec{u}=0$ and $\vec{u}\cdot\vec{n}\big|_{\Gamma}=0$.
This energy dissipation law was first observed by \cite{bodineau2014lyapunov}.
In the case of $\phi(x)=(x-1)^2$, the \eqref{energy1} reduces to the following energy dissipation law with respect to the Pearson $\chi^2$-divergence
\begin{equation}\label{energy2}
\frac{\ud}{\ud t} \int_\Omega \frac{(\rho-\invm)^2}{\invm} \ud { x}
=\frac{\ud}{\ud t} \int_\Omega \frac{\rho^2}{\invm} \ud { x} 
= -2D \int_\Omega \invm \Big| \nabla \frac{\rho}{\invm} \Big|^2 \ud { x}
\leq 0.
\end{equation}
In addition, from the Poincare's inequality in $L^2(\Omega)$ with $u=\frac{\rho}{\invm}-1$, there exists a constant $c$, such that,
\begin{equation}\label{poincare}
\int_{\Omega} |u|^2 \invm \ud { x} \leq c (\invm \nabla u, D \nabla u).
\end{equation}
We obtain
\begin{equation}\label{gap}
\frac{\ud}{\ud t} \int_\Omega \frac{(\rho-\invm)^2}{\invm} \ud { x} 
= -2D \int_\Omega \invm \Big| \nabla \frac{\rho}{\invm} \Big|^2 \ud x  
\leq -\frac{2D}{c} \int_\Omega \frac{(\rho-\invm)^2}{\invm} \ud x .
\end{equation}
Then, by Gronwall's inequality, the above inequality gives the exponential decay from dynamic solution $\rho$ to the equilibrium $\invm$,
\begin{equation}\label{exp_gap}
\int_\Omega {(\frac{\rho}{\invm}-1)^2} \invm \ud x 
= \int_\Omega \frac{(\rho-\invm)^2}{\invm} \ud x  
\leq e^{-\frac{2D}{c} t}\int_\Omega \frac{(\rho^0-\invm)^2}{\invm} \ud x. 
\end{equation}
The ergodicity result above  implies that as long as a positive invariant measure $\invm$ exists, although sometimes its explicit form may not be known, one still can solve it as an equilibrium solution to the original model with no-flux boundary condition \eqref{BC}, i.e., the invariant measure $\mathcal M$ can be obtained from computing the steaty state solution of the following problem 
\begin{align*}
\partial_t\rho = \nabla\cdot(D \nabla{\rho}) - \nabla\cdot(\vec{b}\rho)& && \text{in}~\{t>0\}\times\Omega,\\
\rho = \rho^0& && \text{in}~\{t=0\}\times\Omega,\\
-\rho\vec{b} \cdot \vec{n}  + D \nabla \rho \cdot \vec{n} = 0& && \text{on}~\{t>0\}\times\partial{\Omega}.
\end{align*}
Let us refer the above equation as \textbf{Model~2}, or the conservative form without decomposition.
The vector field $\vec{b}$ is a prescribed time-independent continuously differentiable function. Here, we emphasis $\vec{b}$ can be any general drift field.
In order to produce a non-negative invariant measure, a positivity-preserving numerical scheme for solving \textbf{Model~2} is also preferred.

\subsection{State of the art}
The computational methods for solving the Fokker-Planck equation \eqref{FP-N} or general convection-diffusion equations have been extensively investigated. 
Without being exhaustive, we mention several pioneering studies. 
A well-known positivity-preserving finite volume scheme was proposed by Scharfetter and Gummel for some one-dimensional semiconductor device equations \cite{scharfetter1969large}. 
Extending the Scharfetter-Gummel finite volume scheme to a variety of drift terms, which include the irreversible expression \eqref{FP-N}, and to different boundary conditions, is studied in \cite{markowich1985stationary, markowich1988inverse,  bank1998finite, xu1999monotone, chainais2003finite, bessemoulin2012finite}.  
These numerical schemes are first order accurate structure-preserving and enjoy many good properties, such as preserving positivity and dissipating energy; see also the underlying Markov process structure for upwind schemes in \cite{delarue2011probabilistic, GL21}. 
Besides the accuracy and positivity, when numerically solving Fokker-Planck equations, the large time convergence to the invariant measure is also essential.   The ergodicity for general irreversible process described by \eqref{FP-N} and the equivalence with the corresponding reversible process (with a special drift in gradient form) were proved in \cite{chen2000equivalence}; see also  reviews in \cite{arnold2001convex} and \cite{bodineau2014lyapunov} for the analysis including mixed boundary values. Designing schemes that also preserve the large time convergence to the invariant measure, particularly for the general irreversible process without gradient structure, has been attracting lots of attentions, e.g., see \cite{li2020large} for Fokker-Planck equation in the whole space, and see \cite{filbet2017finite, chainais2020large} and the references therein for boundary-driven convection-diffusion problems. 
To the best knowledge of the authors, high order accurate schemes for solving Fokker-Planck equation \eqref{FP-N} with a general drift field,  which preserve all the desired properties, such as positivity, energy dissipation relation, and particularly the exponential convergence to the invariant measure, are still not available in literatures. High order schemes for Fokker-Planck equation with gradient flow structure was proved in \cite{hu2021positivity} and for generalized Allen-Cahn equation was proved in \cite{shen2021discrete}. 
Although, a comprehensive review on various applications of Fokker-Planck equation brought by irreversible stochastic processes is out of the scope of the present paper, we highlight the general irreversible processes and the processes constructed from numerical schemes for the corresponding Fokker-Planck equations have extensively important applications, which include but not limited to the transition path computations \cite{weinan2006, zhou2016, li2017large,  tao2018hyperbolic, gao2020transition, wei2022optimal} and the accelerated sampling and optimization \cite{turitsyn2011irreversible, duncan2016variance, ZWE, yg20, ye2021efficient}.

\subsection{Main results, methodology, and contributions}
In general, for the Fokker-Planck equations with generic drift terms, it is nontrivial to construct high-order accurate numerical schemes that can preserve all the following structures:
(\emph{i}) mass conservation law \eqref{mass};
(\emph{ii}) energy dissipation relation \eqref{energy2};
(\emph{iii}) well-balancedness, i.e., numerical equilibrium recovers given invariant measure $\invm$; and
(\emph{iv}) ergodicity/spectral gap estimate \eqref{exp_gap}.
In this paper, we focus on constructing and analyzing second and fourth order in space numerical schemes via finite difference implementation of the finite element method for solving Fokker-Planck equations, mainly for \textbf{Model~1}. We will obtain all the good properties (i)-(iv) in the fully discrete second and fourth order schemes for \textbf{Model~1}.
\par
Our algorithms enjoy desired numerical properties. 
Benefiting from the inherent nature of finite element method, the discrete mass conservation law is satisfied naturally.
If the matrix of the linear system in the backward Euler time discretization is a monotone matrix, i.e., its inverse matrix has non-negative entries, then 
we call such schemes {\it monotone} schemes.
We show that the schemes for \textbf{Model~1} are monotone under practical mesh conditions and time step constraints in Section~\ref{sec:monotonicity}, thus both the positivity of the numerical solution and the discrete energy dissipation law for any convex function $f$  hold. Define discrete energy  as
\begin{align}\label{eq:section1_review_energy}
E^n = \sum_{i} \omega_i \invm_i f(\frac{\rho^n_i}{\invm_i}).
\end{align}
Then the quantity $E^n$ is non-increasing with respect to time step $n$. In particular, by selecting the convex function $f(x) = (x-1)^2$, we show discrete Pearson $\chi^2$-divergence energy dissipation law.
Finally, the invariant measure is also recovered {with an exponential convergence rate}. 
For the definition of notation in \eqref{eq:section1_review_energy} and more details on related proofs are shown in Section~\ref{sec:positivity_and_dissipation}.
 
When $\vec u\equiv 0$,  the scheme in this paper for  \textbf{Model 1} reduces to the scheme for the Fokker-Planck equation in \cite{hu2021positivity}, thus the monotonicity discussion of the fourth order scheme is similar to those in \cite{hu2021positivity}. However, due to the extra term $\nabla\cdot{\left(\vec{u}\,\frac{\rho}{\invm}\right)}$ in \textbf{Model 1}, the monotonicity  discussion 
in Section~\ref{sec:monotonicity} is not only necessary but also nontrivial. More importantly, the mesh size and time step constraints for monotonicity in Section~\ref{sec:monotonicity} are simpler than those in \cite{hu2021positivity}, even for the case $\vec u\equiv 0$, which is another contribution of this paper. 

\subsection{Organization of the paper}
The rest of this paper is organized as follows. 
In Section~\ref{sec:numerical_scheme}, we introduce our numerical schemes, which are constructed by finite difference implementation of $Q^1$ and $Q^2$ continuous finite element methods.
In Section~\ref{sec:monotonicity}, we show the monotonicity of our second-order and fourth-order schemes in one and two dimension. 
The system matrices from our fourth order schemes no longer hold the M-matrix structure, however, we still obtain the monotonicity under simple sufficient conditions. 
The structure-preserving properties are discussed in Section~\ref{sec:positivity_and_dissipation}. 
Numerical experiment validations are in Section~\ref{sec-test}.  
Concluding remarks are given in Section~\ref{sec-remark}.

\section{The numerical schemes}\label{sec:numerical_scheme}

Consider a rectangular computational domain $\Omega\subset\bR^d$ ($d=1$ or $2$) with unit outward normal $\vec{n}$. 
Uniformly partition the time interval $[0,T]$ into $N_\mathrm{st}$ subintervals. Let $\Delta{t} = T/N_\mathrm{st}$ denote the time step size.
We discretize two model problems in the previous section with the only unknown density $\rho$.

\subsection{A first order accurate time discretization}
Let $\rho^n$ denote the solution at time step $n$. For Model~1 with a given incompressible $\vec u$ satifying $\vec u\cdot\vec n=0$ along the boundary $\partial\Omega$ and a given invariant measure $\mathcal M$, we consider the following  backward Euler time discretization:
\begin{equation}\label{ttm1}
\rho^{n+1} =\rho^{n}+\Delta t \,\nabla\cdot\left(D\invm\nabla{\frac{\rho^{n+1}}{\invm}}\right) +\Delta t\, \nabla\cdot{\left(\vec{u}\,\frac{\rho^{n+1}}{\invm}\right)}.
\end{equation}
For convenience, introduce an auxiliary variable $g:=\frac{\rho}{\mathcal M}$, then $g^{n+1}$ can be computed as follows:
\begin{subequations}
\label{eq:model1_time_model1}
\begin{align}
\invm g^{n+1} - \Delta{t}\,\nabla\cdot(D \invm \nabla{g^{n+1}}) - \Delta{t}\,\nabla\cdot{(\vec{u}g^{n+1})} = \invm g^{n} \quad\text{in}~\Omega,\label{eq:model1_time1}\\
D\invm \nabla{g^{n+1}}\cdot\vec{n} = 0, \quad \vec{u}\cdot\vec{n}=0 \quad\text{on}~\partial\Omega.\label{eq:model1_time2}
\end{align}  
\end{subequations}
For solving Model~2 with given $\vec b$, the same backward Euler time discretization is given as
\begin{subequations}\label{eq:model1_time_model2}
\begin{align}
\rho^{n+1} - \Delta{t}\,\nabla\cdot(D \nabla{\rho}^{n+1}) + \Delta{t}\,\nabla\cdot(\vec{b}\rho^{n+1}) = \rho^n  \quad\text{in}~\Omega,\label{eq:model2_time1}\\
-\rho^{n+1}\vec{b} \cdot \vec{n} + D \nabla \rho^{n+1} \cdot \vec{n} = 0  \quad\text{on}~\partial\Omega.\label{eq:model2_time2}
\end{align}
\end{subequations}
Although we do not have explicit formula for $\invm$ in the most general case Model 2, for compact domain, we still know the existence of such an invariant measure and thus we can use $\invm$ to obtain the stability and energy dissipation relation. Similar to the derivation of \eqref{pi-decom},  
we recast \eqref{eq:model2_time1} as
\begin{equation}
\rho^{n+1} - \Delta{t}\,\nabla\cdot(D \invm \nabla \frac{{\rho}^{n+1}}{\invm}) - \Delta{t}\,\nabla\cdot(\vec{u}\frac{\rho^{n+1}}{\invm}) = \rho^n \quad\text{in}~\Omega,
\end{equation}
which is exactly the same as \eqref{ttm1}. Therefore, thanks to the existence of $\invm $ and the fact that $\invm $-decomposition reduces Model 2 to Model 1, the stability analysis and energy dissipation law for the backward Euler scheme of  \textbf{Model~2} can also be derived, with an unknown function $\invm$.
We have the following proposition for the time discretization \eqref{ttm1} and \eqref{eq:model1_time_model2}.
\begin{prop}\label{prop_time}
Let the invariant measure satisfying \eqref{pi} and \eqref{BCm} be $\invm$, which has no explicit formula for Model~2. The backward Euler time discretization \eqref{ttm1} for \textbf{Model~1} and \eqref{eq:model1_time_model2} for \textbf{Model~2} satisfy the stability estimate
 \begin{equation}\label{stability}
 \frac{1}{\Delta{t}}\int_\Omega \frac{(\rho^{n+1})^2}{\invm} 
- \frac{1}{\Delta{t}}\int_\Omega \frac{(\rho^{n})^2}{\invm}
\leq - 2D\,\int_\Omega \invm\left|\nabla{\,\frac{\rho^{n+1}}{\invm}}\right|^2,
 \end{equation}
 and the exponential decay of energy 
 \begin{equation}\label{exp_tn_00}
\int_\Omega \frac{(\rho^{n} - \invm)^2}{\invm} \leq (1-\beta)^n \int_\Omega \frac{(\rho^{0} - \invm)^2}{\invm} \leq e^{-\beta n}\int_\Omega \frac{(\rho^{0} - \invm)^2}{\invm},
\end{equation}
where $\beta = \frac{2D \Delta t}{c+ 2D \Delta t}>0.$
 \end{prop}
 \begin{proof}
 First, since the invariant measure $\invm$ always exists due to the ergodicity of the continuous Fokker-Planck equation on a compact domain, the \eqref{eq:model1_time_model2} can be recast as \eqref{ttm1}; as we mentioned above. 
 
 Second, for \eqref{ttm1} (i.e., \eqref{eq:model1_time_model1} in terms of $g^{n+1}=\frac{\rho^{n+1}}{\invm}$,  we prove the stability estimate \eqref{stability}.  Multiplying \eqref{eq:model1_time1} by $g^{n+1}$ and integrating in $\Omega$, we have
 \begin{equation}
 \int_\Omega \invm |g^{n+1}|^2 + \Delta t \la D\invm \nabla g^{n+1}, \nabla g^{n+1} \ra = \int_\Omega  \invm g^n g^{n+1}  \leq \int_\Omega \invm \bbs{\frac{|g^{n+1}|^2}{2} + \frac{|g^{n}|^2}{2}}. 
\end{equation}
Here in the first equality, we used \eqref{antiT} and integration by parts with boundary condition \eqref{bcbc}.  Therefore, we obtain
\begin{equation}
\int_\Omega \invm |g^{n+1}|^2 + 2D \Delta t \int_\Omega \invm |\nabla g^{n+1}|^2  \leq \int_\Omega \invm |g^{n}|^2, 
\end{equation} 
and thus \eqref{stability} holds.

Third, from the mass conservation $\int_\Omega \rho^0=\int_\Omega \mathcal M $, we know $\int \frac{\rho^2}{\invm} dx -1 = \int \frac{(\rho-\invm)^2}{\invm} dx.$ and thus \eqref{stability} becomes
 \begin{equation} 
 \frac{1}{\Delta{t}}\int_\Omega \frac{(\rho^{n+1}-\invm)^2}{\invm} 
- \frac{1}{\Delta{t}}\int_\Omega \frac{(\rho^{n}-\invm)^2}{\invm}
\leq - 2D\,\int_\Omega \invm\left|\nabla{\,\frac{\rho^{n+1}}{\invm}}\right|^2.
 \end{equation}
Then by Poincare's inequality \eqref{poincare} and the inequality in \eqref{gap}, we obtain
\begin{equation}
\frac{1}{\Delta{t}}\int_\Omega \frac{(\rho^{n+1}-\invm)^2}{\invm} 
- \frac{1}{\Delta{t}}\int_\Omega \frac{(\rho^{n}-\invm)^2}{\invm} \leq -\frac{2D}{c} \int_\Omega \frac{(\rho^{n+1}-\invm)^2}{\invm}.
\end{equation}
 Repeating
$\int_\Omega \frac{(\rho^{n} - \invm)^2}{\invm} \leq \frac{1}{1+ \frac{2D}{c}\Delta t} \int_\Omega \frac{(\rho^{n-1} - \invm)^2}{\invm}$ for $n$ times we obtain \eqref{exp_tn_00}.
 \end{proof}
   
   On the other hand, without the explicit expression of $\invm$, we cannot obtain similar theoretical results for the full discretization of Model 2 beyond the positivity and mass conservation.   Without having any explicit information on the invariant measure, it is  challenging to  capture a steady solution by directly solving the dynamic equation in Model 2. For the above two reasons, in the rest of the paper, we only focus on the analysis of the scheme \eqref{eq:model1_time_model1} for Model 1.

\subsection{The continuous $Q^k$ finite element method for spatial derivatives}
Given the function $g^n$, the semi-discrete scheme \eqref{eq:model1_time_model1} is a variable coefficient elliptic equation for $g^{n+1}$ with homogeneous Neumann boundary conditions. 
Let $(\cdot,\cdot)$ denote the $L^2$ inner product on $\Omega$. After multiplying a test function $\phi \in H^1(\Omega)$ and integration by parts, the equivalent variational formulation for the unknown $g^{n+1}\in H^1(\Omega)$ to the equation \eqref{eq:model1_time_model1} can be written as:
\[(\invm g^{n+1}, \phi) 
+ \Delta{t}\,(D\invm\nabla{g^{n+1}}, \nabla{\phi})
+ \Delta{t}\,(\vec{u} g^{n+1}, \nabla{\phi}) 
= (\invm g^{n}, \phi), 
\quad \forall \phi \in H^1(\Omega).\]

Next we consider the finite element method for the spatial operators. 
Let $\mathcal{E}_h = \{E\}$ denote a uniform rectangular partition of the rectangular computational domain $\Omega$.
For any integer $k\geq1$, let $\mathbb{Q}^k(E)$ be the space of tensor product polynomials of degree at most $k$. As an example, for two dimensions,
\begin{align*}
\mathbb{Q}^k(E) = \left\{p(x,y):~ p(x,y) = \sum_{i=0}^k\sum_{j=0}^k p_{ij}x^i y^j,~ (x,y)\in E\right\}.
\end{align*}
The continuous piecewise $\mathbb{Q}^k$ polynomial space $V^h\subset H^1(\Omega)$ is defined by
\begin{align*}
V^h = \{v_h\in C(\Omega):~ v_h|_{E}\in \mathbb{Q}^k(E),~\forall E\in\mathcal{E}_h\}.
\end{align*}
Given $g_h^n \in V^h$, 
the $Q^k$ finite element method is to find $g_h^{n+1}\in V^h$ satisfying 
\begin{align}\label{eq:Model1_varform}
(\invm g_h^{n+1}, \phi_h) 
+ \Delta{t}\,(D\invm\nabla{g_h^{n+1}}, \nabla{\phi_h})
+ \Delta{t}\,(\vec{u} g_h^{n+1}, \nabla{\phi_h}) 
= (\invm g_h^{n}, \phi_h), 
\quad \forall \phi_h \in V^h.
\end{align} 
 
Denote $\rho_h^n = \invm g_h^n$. 
The finite element scheme \eqref{eq:Model1_varform}  has the following properties.
\begin{prop}
For any $n\geq0$, the total mass is conserved in the scheme \eqref{eq:Model1_varform}, namely
\begin{align}\label{eq:dis_mass_conv}
\int_\Omega \rho^n_h = \int_\Omega \rho^0_h.
\end{align}
\end{prop}
\begin{proof}
Take $\phi_h = 1$ in \eqref{eq:Model1_varform}, then for any $n\geq0$, the identity $(\invm g_h^{n+1}, 1) = (\invm g_h^{n}, 1)$ holds. 
\end{proof}
\noindent 
\begin{prop}
For any $n\geq0$, the following two inequalities hold in  the scheme \eqref{eq:Model1_varform},
\begin{subequations}
\begin{align}
\frac{1}{\Delta{t}}\int_\Omega \frac{(\rho^{n+1}_h)^2}{\invm} 
- \frac{1}{\Delta{t}}\int_\Omega \frac{(\rho^{n}_h)^2}{\invm}
\leq - 2D\,\int_\Omega \invm\left|\nabla{\,\frac{\rho^{n+1}_h}{\invm}}\right|^2, \label{eq:dis_energy_ineq1}\\
\frac{1}{\Delta{t}}\int_\Omega \frac{(\rho^{n+1}_h - \invm)^2}{\invm} 
- \frac{1}{\Delta{t}}\int_\Omega \frac{(\rho^{n}_h - \invm)^2}{\invm}
\leq - 2D\,\int_\Omega \invm\left|\nabla{\,\frac{\rho^{n+1}_h}{\invm}}\right|^2. \label{eq:dis_energy_ineq2}
\end{align}
\end{subequations}
\end{prop}
\begin{proof}
Consider $\invm>0$ is time independent and $D>0$ is constant. For any $\phi_h \in V^h$, let us rewrite \eqref{eq:Model1_varform} into the following form
\begin{align}\label{eq:dis_energy_1}
(\sqrt{\invm} g_h^{n+1} - \sqrt{\invm} g_h^{n}, \sqrt{\invm} \phi_h) 
+ \Delta{t}\,(\vec{u} g_h^{n+1}, \nabla{\phi_h}) 
= - \Delta{t}D\,(\sqrt{\invm}\nabla{g_h^{n+1}}, \sqrt{\invm}\nabla{\phi_h}).
\end{align}
Take $\phi_h = g_h^{n+1}\in V^h$ in \eqref{eq:dis_energy_1}. By formula $(a-b)a = \frac{1}{2}a^2 - \frac{1}{2}b^2 + \frac{1}{2}(a-b)^2$, the first term on the left-hand side of above becomes
\begin{align}\label{eq:dis_energy_2}
(\sqrt{\invm} g_h^{n+1} - \sqrt{\invm} g_h^{n}, \sqrt{\invm} g_h^{n+1})
&= \frac{1}{2}(\sqrt{\invm} g_h^{n+1},\sqrt{\invm} g_h^{n+1}) 
- \frac{1}{2}(\sqrt{\invm} g_h^{n}, \sqrt{\invm} g_h^{n})\nonumber\\
&+ \frac{1}{2}(\sqrt{\invm} g_h^{n+1} - \sqrt{\invm} g_h^{n}, \sqrt{\invm} g_h^{n+1} - \sqrt{\invm} g_h^{n}).
\end{align}
For the second term on the left-hand side of \eqref{eq:dis_energy_1}, we have
\begin{align*}
\Delta{t}\,(\vec{u} g_h^{n+1}, \nabla{g_h^{n+1}})
= \Delta{t}\,\int_\Omega \vec{u} \cdot g_h^{n+1}\nabla{g_h^{n+1}}
= \frac{\Delta{t}}{2}\,\int_\Omega \vec{u}\cdot\nabla{|g_h^{n+1}|^2}.
\end{align*}
Since the field $\vec{u}$ is incompressible, we have
\begin{align*}
\Delta{t}\,(\vec{u} g_h^{n+1}, \nabla{g_h^{n+1}})
= \frac{\Delta{t}}{2}\,\int_\Omega \Big(\vec{u}\cdot\nabla{|g_h^{n+1}|^2} + (\nabla\cdot{\vec{u}})|g_h^{n+1}|^2\Big)
= \frac{\Delta{t}}{2}\,\int_\Omega \nabla\cdot{(\vec{u}\,|g_h^{n+1}|^2)}.
\end{align*}
By condition $\vec{u}\cdot\vec{n} = 0$, see the last equation in \eqref{model1}, we have
\begin{align}\label{eq:dis_energy_3}
\Delta{t}\,(\vec{u} g_h^{n+1}, \nabla{g_h^{n+1}})
= \frac{\Delta{t}}{2}\,\int_{\partial{\Omega}} |g_h^{n+1}|^2\,(\vec{u}\cdot\vec{n}) = 0.
\end{align}
Thus, select $\phi_h = g_h^{n+1}\in V^h$, substitute \eqref{eq:dis_energy_2} and \eqref{eq:dis_energy_3} into \eqref{eq:dis_energy_1}, we get the following inequality  
\begin{align}\label{eq:dis_energy_4}
&\frac{1}{2}(\sqrt{\invm} g_h^{n+1},\sqrt{\invm} g_h^{n+1}) 
- \frac{1}{2}(\sqrt{\invm} g_h^{n}, \sqrt{\invm} g_h^{n}) \nonumber\\
=&\, - \Delta{t}D\,(\sqrt{\invm}\nabla{g_h^{n+1}}, \sqrt{\invm}\nabla{g_h^{n+1}})
-\frac{1}{2}(\sqrt{\invm} g_h^{n+1} - \sqrt{\invm} g_h^{n}, \sqrt{\invm} g_h^{n+1} - \sqrt{\invm} g_h^{n})
\nonumber\\
\leq&\, - \Delta{t}D\,(\sqrt{\invm}\nabla{g_h^{n+1}}, \sqrt{\invm}\nabla{g_h^{n+1}}).
\end{align}
Recall that $\rho^{n+1}_h = \invm g_h^{n+1}$, which is equivalent to $g_h^{n+1} = \rho^{n+1}_h/\invm$. Multiply \eqref{eq:dis_energy_4} by $2/\Delta{t}$ on both side, we obtain \eqref{eq:dis_energy_ineq1}.
Finally, from mass conservation \eqref{eq:dis_mass_conv}, we also have \eqref{eq:dis_energy_ineq2} holds.
\end{proof}
\begin{prop}\label{thm:exp_t}
For any $n\geq0$, the following inequality hold
\begin{align}
\frac{1}{\Delta{t}}\int_\Omega \frac{(\rho^{n+1}_h - \invm)^2}{\invm} 
- \frac{1}{\Delta{t}}\int_\Omega \frac{(\rho^{n}_h - \invm)^2}{\invm}
\leq - \frac{2D}{c}\,\int_\Omega \frac{(\rho^{n+1}_h - \invm)^2}{\invm}. \label{eq:dis_energy_ineq3}
\end{align}
Consequently, we have the exponential decay of $\rho^{n+1}_h$
\begin{equation}\label{exp_tn}
\int_\Omega \frac{(\rho^{n}_h - \invm)^2}{\invm} \leq (1-\beta)^n \int_\Omega \frac{(\rho^{0}_h - \invm)^2}{\invm} \leq e^{-\beta n}\int_\Omega \frac{(\rho^{0}_h - \invm)^2}{\invm},
\end{equation}
where $\beta = \frac{2D \Delta t}{c+ 2D \Delta t}>0.$
\end{prop}
The proof of this proposition is identically same as \eqref{exp_tn_00} in Proposition \ref{prop_time}.

\subsection{The finite difference implementation}

For implementing the finite element method above, usually quadrature is used for computing the integrals. 
On the other hand,
it is well known that a finite element method with suitable quadrature is also a finite difference scheme. 
When $(k+1)$-point Gauss-Lobatto quadrature is used for \eqref{eq:Model1_varform}, the scheme is also referred to as $Q^k$ spectral element method \cite{maday1990optimal}, and it can be regarded as a $(k+2)$-th order accurate finite difference scheme with respect to the discrete $\ell^2$-norm at quadrature points for $k\geq 2$. See \cite{li2021accuracy,li2020superconvergence} for rigorous a priori error estimates.

Any $\mathbb{Q}^k$ polynomial on rectangular element can be represented as a Lagrangian interpolation polynomial at $(k+1)^d$ Gauss-Lobatto points. Thus, these points are not only quadrature nodes but also representing all degrees of freedom. 
In addition, for $k\leq2$, all Gauss-Lobatto points form a uniform grid.  
Let $\langle\cdot,\cdot\rangle$ denote the inner product $(\cdot,\cdot)$ evaluated by Gauss-Lobatto quadrature. Then, after replacing all integrals in \eqref{eq:Model1_varform} by Gauss-Lobatto quadrature, the scheme becomes:
\begin{align}\label{eq:Model1_quadform}
\langle\invm g_h^{n+1}, \phi_h\rangle 
+ \Delta{t}\,\langle D\invm\nabla{g_h^{n+1}}, \nabla{\phi_h}\rangle
+ \Delta{t}\,\langle\vec{u} g_h^{n+1}, \nabla{\phi_h}\rangle 
= \langle\invm g_h^{n}, \phi_h\rangle, 
\quad \forall \phi_h \in V^h.
\end{align} 

In particular, we can obtain a fourth order accurate finite difference scheme when using $Q^2$ element with $3$-point Gauss-Lobatto quadrature in \eqref{eq:Model1_varform}. For the $Q^1$ element method with $2$-point Gauss-Lobatto quadrature in \eqref{eq:Model1_varform}, we get a second order accurate finite difference scheme, which is exactly the same as the centered difference at interior grid points. These two finite difference schemes can be proved monotone for convection-diffusion operators \cite{li2020monotonicity, hu2021positivity, shen2021discrete}, thus positivity-preserving and energy decaying. 
In this paper, we  only consider these two finite difference schemes. 
\subsection{The second order scheme in one dimension}
For $\Omega=[-L,L]$, consider uniform grid points with spacing $h$, $-L = x_1 < x_2 < \cdots < x_N = L$. The mesh $\mathcal{E}_h$  consists of intervals $[x_i, x_{i+1}],~ i = 1, \cdots, N-1.$
 
 Following the discussion in \cite{hu2021positivity}, it is straightforward to verify that the scheme \eqref{eq:Model1_quadform} can be equivalently written in the following finite difference form:
 
\[ \invm_1 g^{n+1}_1
-\Delta{t}\, \frac{u_1 g_1^{n+1} + u_2 g_2^{n+1}}{h}
+\Delta{t}\,
\frac{D(\invm_{1} + \invm_{2})g_{1}^{n+1} 
- D(\invm_{1} + \invm_{2})g_{2}^{n+1}}{h^2}
= \invm_1 g^{n}_1;
\]
for $i=2, \cdots, N-1$, 
\[\label{scheme-1d2nd-interior}
\resizebox{.99\hsize}{!}{$
 \invm_i g_i^{n+1}
+\Delta{t}\, \frac{u_{i-1} g_{i-1}^{n+1} - u_{i+1} g_{i+1}^{n+1}}{2h}
+\Delta{t}\, 
\frac{-D(\invm_{i-1} + \invm_i)g_{i-1}^{n+1} 
+ D(\invm_{i-1} + 2\invm_i + \invm_{i+1})g_{i}^{n+1} 
- D(\invm_{i} + \invm_{i+1})g_{i+1}^{n+1}}{2h^2}
= \invm_i g_i^{n},$}
\]and
\[ \invm_N g^{n+1}_N
+\Delta{t} \frac{u_{N-1} g_{N-1}^{n+1} + u_N g_N^{n+1}}{h}
+\Delta{t}
\frac{-D(\invm_{N-1} + \invm_{N})g_{N-1}^{n+1} 
+ D(\invm_{N-1} + \invm_{N})g_{N}^{n+1}}{h^2}
= \invm_N g^{n}_N.
\]
For convenience, we introduce some ghost point values defined as 
$$g_0: =g_2,\quad  g_{N+1}:=g_{N-1}, \mathcal M_0:=\mathcal M_{2},\mathcal M_{N+1}:=\mathcal M_{N-1}$$
and 
$$u_{0}:=-u_2,\quad u_{N+1}:=-u_{N-1}.$$
We emphasize that ghost point values are only used to simplify the representation of the scheme, and they are not needed or used in the implementation. 
Recall that the velocity field in Model~1 satisfies boundary condition $\vec u\cdot \vec n=0$, thus $u_1=u_N=0$.
With the ghost point value notation and the boundary condition $u_1=u_N=0$, it is straightforward to see that the scheme above can 
be written as 
\begin{eqnarray}
\notag
 \resizebox{.99\hsize}{!}{$
 \invm_i g_i^{n+1}
+\Delta{t}\, \frac{u_{i-1} g_{i-1}^{n+1} - u_{i+1} g_{i+1}^{n+1}}{2h}
+\Delta{t}\, 
\frac{-D(\invm_{i-1} + \invm_i)g_{i-1}^{n+1} 
+ D(\invm_{i-1} + 2\invm_i + \invm_{i+1})g_{i}^{n+1} 
- D(\invm_{i} + \invm_{i+1})g_{i+1}^{n+1}}{2h^2}
= \invm_i g_i^{n},$}\\
 \label{scheme-1d2nd} \forall i=1, \cdots, N.\quad
\end{eqnarray}

 The finite difference scheme  \eqref{scheme-1d2nd} is obtained from finite element method with quadrature, and its second order accuracy is trivially implied by standard finite element error estimates. At domain interior points, the scheme \eqref{scheme-1d2nd-interior} is the same as the traditional second order centered difference scheme. However, the traditional centered difference scheme would give a different boundary scheme, for which the second order accuracy is quite difficult to justify due to the first order truncation error at boundaries. See Remark 3.3 in \cite{hu2021positivity} for more details. 

\subsection{The second order scheme in two dimensions}
For a square domain $\Omega=[-L,L]^2$, we consider a uniform grid with spacing $h$ consisting of $-L = x_1 < x_2 < \cdots < x_N = L$ and $-L = y_1 < y_2 < \cdots < y_N = L$. The mesh $\mathcal{E}_h$ consists of  $[x_i, x_{i+1}]\times[y_j, y_{j+1}],~ i,j = 1,\cdots,N-1$.
We employ the abbreviation $\invm_{i,j} = \invm(x_i,y_j)$ for a function. 
With similarly defined ghost point values and the boundary condition for velocity $\vec u\cdot \vec n=0$,
 the scheme finite difference form  of \eqref{eq:Model1_quadform}  with $Q^1$ element and $2\times 2$ 
Gauss-Lobatto quadrature is given as 
\begin{align}
\notag
& \invm_{i,j}g_{ij}^{n+1}
+\Delta{t}\, \frac{u_{i-1,j} g_{i-1,j}^{n+1} - u_{i+1,j} g_{i+1,j}^{n+1}}{2h}
+\Delta{t}\, \frac{v_{i,j-1} g_{i,j-1}^{n+1} - v_{i,j+1} g_{i,j+1}^{n+1}}{2h}\\
\notag
+&\Delta{t}\, 
\frac{-D(\invm_{i-1,j} + \invm_{i,j})g_{i-1,j}^{n+1} 
+ D(\invm_{i-1,j} + 2\invm_{i,j} + \invm_{i+1,j})g_{ij}^{n+1} 
- D(\invm_{i,j} + \invm_{i+1,j})g_{i+1,j}^{n+1}}{2h^2}\\
\notag
+&\Delta{t}\, 
\frac{-D(\invm_{i,j-1} + \invm_{i,j})g_{i,j-1}^{n+1} 
+ D(\invm_{i,j-1} + 2\invm_{i,j} + \invm_{i,j+1})g_{ij}^{n+1} 
- D(\invm_{i,j} + \invm_{i,j+1})g_{i,j+1}^{n+1}}{2h^2}\\
=&\, \invm_{i,j} g_{ij}^{n},\qquad \forall i,j=1,\cdots, N. \label{2d-2nd-scheme}
\end{align}

\subsection{The fourth order scheme in one dimension} 
 Assume the domain $\Omega=[-L, L]$ is partitioned into $k$ uniform intervals with cell length $2h$. 
Then all $3$-point Gauss-Lobatto points for each small interval form 
an uniform grid  $-L=x_1<x_2<\cdots<x_N=L$ with grid spacing $h$ and $N=2k+1.$ Thus the number of grid points for this fourth order scheme must be odd. The mesh $\mathcal{E}_h$  consists of intervals $[x_{i}, x_{i+2}],~ i = 1, 3, \cdots, N-2$.

 Following the discussion in \cite{li2020superconvergence, hu2021positivity, shen2021discrete}, it is straightforward to verify that the scheme \eqref{eq:Model1_quadform} with $P^2$ element and 3-point Gauss-Lobatto quadrature can be equivalently written in the following finite difference form:
 \[ \resizebox{.99\hsize}{!}{$  \invm_1 g_1^{n+1}
+ \Delta{t} \frac{-3u_1g_1^{n+1} - 4u_2 g_{2}^{n+1} + u_3 g_{3}^{n+1}}{2h}
+ \Delta{t}\,\frac{D(9\invm_1 + 4\invm_{2} + \invm_{3})g_1^{n+1} - D(12\invm_{1} + 4\invm_3)g_{2}^{n+1}}{4h^2}
+ \Delta{t}\,\frac{D(3\invm_{1} - 4\invm_{2} + 3\invm_3)g_{3}^{n+1}}{4h^2}
= \invm_1 g_1^{n}; $} \]
\begin{eqnarray*}
  \resizebox{.99\hsize}{!}{$ \invm_i g_i^{n+1}
+ \Delta{t} \frac{-u_{i-2} g_{i-2}^{n+1} + 4u_{i-1} g_{i-1}^{n+1} - 4u_{i+1} g_{i+1}^{n+1} + u_{i+2} g_{i+2}^{n+1}}{4h}
+ \Delta{t}\,\frac{D(3\invm_{i-2} - 4\invm_{i-1} + 3\invm_i)g_{i-2}^{n+1} - D(4\invm_{i-2} + 12\invm_i)g_{i-1}^{n+1}}{8h^2}
$}\\
  \resizebox{.99\hsize}{!}{$ +\Delta{t}\,\frac{D(\invm_{i-2} + 4\invm_{i-1} + 18\invm_i + 4\invm_{i+1} + \invm_{i+2})g_i^{n+1}}{8h^2}
+ \Delta{t}\,\frac{-D(12\invm_i + 4\invm_{i+2})g_{i+1}^{n+1} + D(3\invm_i - 4\invm_{i+1} + 3\invm_{i+2})g_{i+2}^{n+1}}{8h^2}
= \invm_i g_i^{n}, $} \\
\forall i=3,5,\cdots, N-2 \quad (\mbox{corresponding to interior knots});
\end{eqnarray*}
\begin{eqnarray*}
  \resizebox{.99\hsize}{!}{$ \invm_i g_i^{n+1}
+\Delta{t}\, \frac{u_{i-1} g_{i-1}^{n+1} - u_{i+1} g_{i+1}^{n+1}}{2h}
+ \Delta{t}\, \frac{-D(3\invm_{i-1}+\invm_{i+1})g_{i-1}^{n+1} + 4D(\invm_{i-1}+\invm_{i+1})g_i^{n+1} - D(\invm_{i-1}+3\invm_{i+1})g_{i+1}^{n+1}}{4h^2}
= \invm_i g_i^n, $}\\
\forall  i=2, 4, \cdots N-1\quad  (\mbox{corresponding to interval midpoints});
\end{eqnarray*}
 \[ \resizebox{.99\hsize}{!}{$ \invm_N g_N^{n+1}
+ \Delta{t} \frac{- u_{N-2} g_{N-2}^{n+1} + 4u_{N-1} g_{N-1}^{n+1} + 3u_N g_N^{n+1}}{2h}
+ \Delta{t}\,\frac{D(3\invm_{N-2} - 4\invm_{N-1} + 3\invm_N)g_{N-2}^{n+1} - D(4\invm_{N-2} + 12\invm_N)g_{N-1}^{n+1}}{4h^2}
+ \Delta{t}\,\frac{D(\invm_{N-2} + 4\invm_{N-1} + 9\invm_N)g_N^{n+1}}{4h^2}
= \invm_N g_N^{n}. $} \]
 
 For convenience, we introduce   ghost point values defined as 
$$g_{-1}:=g_3, g_0: =g_2,  g_{N+1}:=g_{N-1},g_{N+2}:=g_{N-2},$$ 
$$\mathcal M_{-1}:=\mathcal M_{3},\mathcal M_0:=\mathcal M_{2},\mathcal M_{N+1}:=\mathcal M_{N-1},\mathcal M_{N+2}:=\mathcal M_{N-2},$$
and 
$$u_{-1}:=-u_3, u_{0}:=-u_2,\quad u_{N+1}:=-u_{N-1}, u_{N+2}:=-u_{N-2}.$$
 
With the ghost point value notation and the velocity boundary condition $u_1=u_N=0$, the fourth order finite difference scheme above
can be  written as
\begin{subequations}
  \label{1d-4th-scheme}
  \begin{eqnarray}
 \notag \resizebox{.99\hsize}{!}{$ \invm_i g_i^{n+1}
+ \Delta{t} \frac{-u_{i-2} g_{i-2}^{n+1} + 4u_{i-1} g_{i-1}^{n+1} - 4u_{i+1} g_{i+1}^{n+1} + u_{i+2} g_{i+2}^{n+1}}{4h}
+ \Delta{t}\,\frac{D(3\invm_{i-2} - 4\invm_{i-1} + 3\invm_i)g_{i-2}^{n+1} - D(4\invm_{i-2} + 12\invm_i)g_{i-1}^{n+1}}{8h^2}
$}\\
 \notag \resizebox{.99\hsize}{!}{$ +\Delta{t}\,\frac{D(\invm_{i-2} + 4\invm_{i-1} + 18\invm_i + 4\invm_{i+1} + \invm_{i+2})g_i^{n+1}}{8h^2}
+ \Delta{t}\,\frac{-D(12\invm_i + 4\invm_{i+2})g_{i+1}^{n+1} + D(3\invm_i - 4\invm_{i+1} + 3\invm_{i+2})g_{i+2}^{n+1}}{8h^2}
= \invm_i g_i^{n}, $} \\
\forall i=1,5,\cdots, N \quad (\mbox{odd $i$, corresponding to knots});\quad \label{1d-4th-scheme-1}\\
 \notag \resizebox{.99\hsize}{!}{$ \invm_i g_i^{n+1}
+\Delta{t}\, \frac{u_{i-1} g_{i-1}^{n+1} - u_{i+1} g_{i+1}^{n+1}}{2h}
+ \Delta{t}\, \frac{-D(3\invm_{i-1}+\invm_{i+1})g_{i-1}^{n+1} + 4D(\invm_{i-1}+\invm_{i+1})g_i^{n+1} - D(\invm_{i-1}+3\invm_{i+1})g_{i+1}^{n+1}}{4h^2}
= \invm_i g_i^n, $}\\
\forall  i=2, 4, \cdots N-1\quad  (\mbox{even $i$, corresponding to interval midpoints}).\quad \label{1d-4th-scheme-2}
\end{eqnarray}
\end{subequations}

\subsection{The fourth order scheme in two dimensions}

Assume the domain is $\Omega=[-L, L]\times[-L,L]$ with 
an uniform $N\times N$ grid point with spacing $h$, obtained from all $3\times 3$ Gauss-Lobatto points on a uniform rectangular mesh with $k\times k$ cells. Thus $N=2k+1$.    
For the $Q^2$ finite element method on uniform rectangular meshes, there are three types of grid point values as shown in Figure \ref{fig-points}:
\begin{align*}
\text{knot:}& \quad \text{both}~i~\text{and}~j~\text{are odd},\\
\text{edge~center:}& \quad i~\text{is even and}~j~\text{is odd}~\text{(parallel~to~x-axis)}\quad\text{or}\\
~ & \quad i~\text{is odd and}~j~\text{is even}~\text{(parallel~to~y-axis)},\\
\text{cell~center:}& \quad \text{both}~i~\text{and}~j~\text{are even}.
\end{align*}
\begin{figure}[h!]
\begin{center}
 \scalebox{0.6}{
\begin{tikzpicture}[samples=100, domain=-3:3, place/.style={circle,draw=blue!50,fill=blue,thick,
inner sep=0pt,minimum size=1.5mm},transition/.style={circle,draw=red,fill=red,thick,inner sep=0pt,minimum size=2mm}
,point/.style={circle,draw=black,fill=black,thick,inner sep=0pt,minimum size=2mm}]

\draw[color=red] (-2,-2)--(-2,2);
\draw[color=red] (-2,-2)--(2,-2);
\draw[color=red] (2,-2)--(2, 2);
\draw[color=red] (-2,2)--(2, 2);
\node at ( -2,-2) [place] {};
\node at ( 0,-2) [point] {};
\node at ( 2,-2) [place] {};

\node at ( -2,0) [point] {};
\node at ( 0,0) [transition] {};
\node at ( 2,0) [point] {};

\node at ( -2,2) [place] {};
\node at ( 0,2) [point] {};
\node at ( 2,2) [place] {};
 \end{tikzpicture}
}
\caption{Three types of  grid points: red cell center, blue knots and black edge centers for a $Q^2$ finite element cell. }
\end{center}
\label{fig-points}
\end{figure}
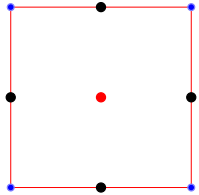

The fourth order finite difference scheme is given as the following:
\begin{subequations}\label{2d-4th-scheme}
\begin{align}
\resizebox{.99\hsize}{!}{$ \invm_{ij}g^{n+1}_{ij}
+\, \Delta{t}\,\frac{-u_{i-2,j}g^{n+1}_{i-2,j} + 4u_{i-1,j}g^{n+1}_{i-1,j} - 4u_{i+1,j}g^{n+1}_{i+1,j} + u_{i+2,j}g^{n+1}_{i+2,j}}{4h}+\, \Delta{t}\,\frac{-v_{i,j-2}g^{n+1}_{i,j-2} + 4v_{i,j-1}g^{n+1}_{i,j-1} - 4v_{i,j+1}g^{n+1}_{i,j+1} + v_{i,j+2}g^{n+1}_{i,j+2}}{4h} $} \nonumber\\  
\resizebox{.99\hsize}{!}{$ +\, \Delta{t}\,\frac{D(3\invm_{i-2,j} - 4\invm_{i-1,j} + 3\invm_{i,j})g^{n+1}_{i-2,j} - D(4\invm_{i-2,j} + 12\invm_{i,j})g^{n+1}_{i-1,j}}{8h^2}
+\, \Delta{t}\,\frac{D(\invm_{i-2,j} + 4\invm_{i-1,j} + 18\invm_{i,j} + 4\invm_{i+1,j} + \invm_{i+2,j})g^{n+1}_{ij}}{8h^2}$} \nonumber\\
\resizebox{.99\hsize}{!}{$ +\, \Delta{t}\,\frac{- D(12\invm_{i,j} + 4\invm_{i+2,j})g^{n+1}_{i+1,j} + D(3\invm_{i,j} - 4\invm_{i+1,j} + 3\invm_{i+2,j})g^{n+1}_{i+2,j}}{8h^2}
+\, \Delta{t}\,\frac{D(3\invm_{i,j-2} - 4\invm_{i,j-1} + 3\invm_{i,j})g^{n+1}_{i,j-2} - D(4\invm_{i,j-2} + 12\invm_{i,j})g^{n+1}_{i,j-1}}{8h^2}$} \nonumber\\
\resizebox{.99\hsize}{!}{$ +\, \Delta{t}\,\frac{D(\invm_{i,j-2} + 4\invm_{i,j-1} + 18\invm_{i,j} + 4\invm_{i,j+1} + \invm_{i,j+2})g^{n+1}_{ij}}{8h^2}
+\, \Delta{t}\,\frac{- D(12\invm_{i,j} + 4\invm_{i,j+2})g^{n+1}_{i,j+1} + D(3\invm_{i,j} - 4\invm_{i,j+1} + 3\invm_{i,j+2})g^{n+1}_{i,j+2}}{8h^2}
=\, \invm_{ij}g^{n}_{ij},$} \nonumber\\
\mbox{if $(x_i,y_j)$ is a knot;
}.
\end{align}
\begin{align}
\resizebox{.99\hsize}{!}{$ 
 \invm_{ij}g^{n+1}_{ij}
+\, \Delta{t}\,\frac{v_{i,j-1}g^{n+1}_{i,j-1}-v_{i,j+1}g^{n+1}_{i,j+1}}{2h}
+\, \Delta{t}\,\frac{-u_{i-2,j}g^{n+1}_{i-2,j} + 4u_{i-1,j}g^{n+1}_{i-1,j} - 4u_{i+1,j}g^{n+1}_{i+1,j} + u_{i+2,j}g^{n+1}_{i+2,j}}{4h}$} \nonumber\\
\resizebox{.99\hsize}{!}{$ 
+\, \Delta{t}\,\frac{D(3\invm_{i-2,j} - 4\invm_{i-1,j} + 3\invm_{i,j})g^{n+1}_{i-2,j} - D(4\invm_{i-2,j} + 12\invm_{i,j})g^{n+1}_{i-1,j}}{8h^2}
+\, \Delta{t}\, \frac{D(\invm_{i-2,j} + 4\invm_{i-1,j} + 18\invm_{ij} + 4\invm_{i+1,j} + \invm_{i+2,j})g^{n+1}_{ij}}{8h^2}$} \nonumber\\
\resizebox{.99\hsize}{!}{$  +\, \Delta{t}\,\frac{-D(12\invm_{i,j} + 4\invm_{i+2,j})g^{n+1}_{i+1,j} + D(3\invm_{i+2,j} - 4\invm_{i+1,j} + 3\invm_{i,j})g^{n+1}_{i+2,j}}{8h^2}
+\, \Delta{t}\, \frac{-D(3\invm_{i,j-1} + \invm_{i,j+1})g^{n+1}_{i,j-1}}{4h^2}$}\nonumber\\
+\, \Delta{t}\, \frac{D(\invm_{i,j-1} + \invm_{i,j+1})g^{n+1}_{ij}}{h^2}
+\, \Delta{t}\, \frac{-D(\invm_{i,j-1} + 3\invm_{i,j+1})g^{n+1}_{i,j+1}}{4h^2}=\, \invm_{ij}g^{n}_{ij},\nonumber\\
\mbox{if $(x_i,y_j)$ is an edge (parallel to $y$-axis) center;}
\end{align}
\begin{align}
 \invm_{ij}g^{n+1}_{ij}
+\, \Delta{t}\,\frac{u_{i-1,j}g^{n+1}_{i-1,j} - u_{i+1,j}g^{n+1}_{i+1,j}}{2h} 
+ \Delta{t}\,\frac{v_{i,j-1}g^{n+1}_{i,j-1} - v_{i,j+1}g^{n+1}_{i,j+1}}{2h} \nonumber\\
\resizebox{.99\hsize}{!}{$  +\, \Delta{t}\, \frac{-D(3\invm_{i-1,j} + \invm_{i+1,j})g^{n+1}_{i-1,j} - D(\invm_{i-1,j} + 3\invm_{i+1,j})g^{n+1}_{i+1,j}}{4h^2}
+\, \Delta{t}\, \frac{D(\invm_{i-1,j}+\invm_{i+1,j}+\invm_{i,j-1}+\invm_{i,j+1})g^{n+1}_{ij}}{h^2}$} \nonumber\\
\resizebox{.99\hsize}{!}{$  +\, \Delta{t}\, \frac{-D(3\invm_{i,j-1} + \invm_{i,j+1})g^{n+1}_{i,j-1} - D(\invm_{i,j-1} + 3\invm_{i,j+1})g^{n+1}_{i,j+1}}{4h^2} 
=\, \invm_{ij}g^{n}_{ij}, \quad
\mbox{if $(x_i,y_j)$ is a cell center.} $}
\end{align}
\end{subequations}
For the grid point $(x_i,y_j)$ which is an edge center for an edge parallel to $x$-axis, 
the scheme  is very similar as above, thus omitted here. 

\section{Monotonicity of the fully discrete schemes}\label{sec:monotonicity} 

In this section, we will prove the monotonicity of the two fully discrete schemes. 

\subsection{The M-matrix structure in the second order scheme}
A matrix $\vecc{A}$ is called monotone if all entries of its inverse are nonnegative, namely, $\vecc{A}^{-1}\geq0$. In this paper, all inequalities for matrices are entry-wise inequalities. 
A square matrix $\vecc{A}$ is called an M-matrix if it can be expressed in the form $\vecc{A} = s\vecc{I} - \vecc{B}$, where $\vecc{B}\geq 0$ and $s$ is greater than the spectral radius of $\vecc{B}$. There are many equivalent definitions or characterizations of M-matrix. 
A comprehensive review of M-matrix can be found in \cite{plemmons1977m}. The nonsingular M-matrix is an inverse-positive matrix and it serves as a convenient tool for proving monotonicity. A sufficient and necessary condition to characteristic nonsingular M-matrix is stated as follows:
\begin{lem}\label{thm:M_matrix1}
For a real square matrix $\vecc{A}$ with positive diagonal entries and nonpositive off-diagonal entries, it is a nonsingular M-matrix if and only if there exists a positive diagonal matrix $\vecc{D}$ such that $\vecc{A}\vecc{D}$ has all positive row sums.
\end{lem}
We also state a sufficient but not necessary condition to verify nonsingular M-matrix; c.f. \cite{plemmons1977m}.
\begin{lem}\label{thm:M_matrix2}
For a real square matrix $\vecc{A}$ with positive diagonal entries and nonpositive off-diagonal entries, it is a nonsingular M-matrix if all the row sums of $\vecc{A}$ are nonnegative and at least one row sum is positive.
\end{lem}

\subsubsection{The second order scheme in one dimension} 
We verify that the matrix in the scheme above satisfies Lemma~\ref{thm:M_matrix2}. 
The following mesh constraint is sufficient for off-diagonal entries in the system matrix to be non-positive: 
 \begin{equation}
h |u_j| \leq D\min \{\invm_{j-1}, \invm_{j}, \invm_{j+1}\}, \quad \forall j. \label{eq:1D2order_suf1}
\end{equation}
To guarantee the nonnegative row sums of the system matrix with at least one strictly positive row sum, the following constraints on time step size are sufficient:
\begin{equation}
\invm_i + \Delta{t}\,\frac{u_{i-1}- u_{i+1}}{2h} > 0 \quad \forall i.  \label{eq:1D2order_suf2} 
\end{equation}
Recall that the velocity field $\vec u$ is incompressible in Model~1, thus $u(x)\equiv C$ in one dimension. So \eqref{eq:1D2order_suf2} is trivially satisfied for positive measure $\mathcal M$. So we have the following result. 
\begin{thm}
 Under the mesh and time step constraints \eqref{eq:1D2order_suf1} and \eqref{eq:1D2order_suf2}, the coefficient matrix for the unknown vector $g^{n+1}$ in 
 the second order finite difference scheme \eqref{scheme-1d2nd} forms an M-matrix and thus is monotone.   
 In particular, in one dimension,   discrete divergence free velocity field $u$ is constant and  the second order finite difference scheme \eqref{scheme-1d2nd} is monotone
 under the mesh constraint \eqref{eq:1D2order_suf1}. 
\end{thm}

\subsubsection{The second order scheme in two dimensions}

Next we verify that the matrix in the scheme \eqref{2d-2nd-scheme} satisfies Lemma~\ref{thm:M_matrix2}. 
The following mesh constraint is sufficient for off-diagonal entries in the system matrix to be non-positive:  
\begin{align}
h |u_{i,j}| &\leq D\min \{\invm_{i-1,j}, \invm_{i+1, j}, \invm_{i,j}, \invm_{i,j-1}, \invm_{i,j+1}\}, \quad \forall i, j. \label{eq:2D2order_suf1}
\end{align}
To guarantee the nonnegative row sums of the system matrix with at least one strictly positive row sum, the following constraints on time step size are sufficient:
\begin{align}
 \invm_{i,j}
+\Delta{t}\, \frac{u_{i-1,j} - u_{i+1,j} }{2h}
+\Delta{t}\, \frac{v_{i,j-1}  - v_{i,j+1} }{2h}>  0.  \label{eq:2D2order_suf2}
\end{align}
Notice that \eqref{eq:2D2order_suf2} is trivially satisfied for positive measure $\mathcal M$, if the following discrete divergence free constraint
is satisfied
\begin{equation}
  \frac{u_{i-1,j} - u_{i+1,j} }{2h}
+  \frac{v_{i,j-1}  - v_{i,j+1} }{2h}=0.
\label{2d-disccrete_div_free}
\end{equation}
Recall that the velocity field $\vec u$ is incompressible in Model~1, thus one can preprocess the given $\vec u$ such that the velocity point values satisfies
\eqref{2d-disccrete_div_free}. So we have the following result. 
\begin{thm}
 Under the mesh and time step constraints \eqref{eq:2D2order_suf1} and \eqref{eq:2D2order_suf2}, the coefficient matrix for the unknown vector $g^{n+1}$ in 
 the second order finite difference scheme \eqref{2d-2nd-scheme} forms an M-matrix thus is monotone.   
 In particular, with a discrete divergence free velocity field satisfying \eqref{2d-disccrete_div_free}, the matrix in second order finite difference scheme \eqref{2d-2nd-scheme} is monotone under the mesh constraint \eqref{eq:2D2order_suf1}. 
\end{thm}

 \subsection{Lorenz's sufficient condition for monotonicity}
 In general, M-matrix structure is only a very conveneint condition for verifying monotonicity, rather than a necessary condition. 
 Moreover, almost all high order schemes simply do not have any M-matrix structure due to positive off-diagonal entries. 
In \cite{lorenz1977inversmonotonie}, Lorenz proposed a convenient sufficient condition for a matrix to be a product of M-matrices. We review Lorenz's sufficient condition in this subsection. See also \cite{li2020monotonicity} for a review. 

Let matrix $\vecc{A}_d$ be a diagonal matrix denoting the diagonal part of $\vecc{A} = [a_{ij}]\in\bR^{n\times n}$ and $\vecc{A}_a = \vecc{A}-\vecc{A}_d$. We further decompose $\vecc{A}_a$ into positive and negative off-diagonal parts. More precisely, we define:
\begin{align*}
\vecc{A}_d =
\begin{cases}
a_{ii}, & \text{if}~i=j,\\
0,      & \text{if}~i\neq j,
\end{cases} &&
\vecc{A}_a^+ =
\begin{cases}
a_{ij}, & \text{if}~a_{ij}>0,~i\neq j,\\
0,      & \text{otherwise},
\end{cases} &&
\vecc{A}_a^- = \vecc{A}_a - \vecc{A}_a^+.
\end{align*}
 
\begin{defn}
Let $\mathcal{N} = \{1,2,\dots,n\}$. For $\mathcal{N}_1,~\mathcal{N}_2\subset\mathcal{N}$, we say a matrix $\vecc{A}$ of size $n\times n$ connects $\mathcal{N}_1$ and $\mathcal{N}_2$, if
\begin{align}\label{eq:connect}
\forall i_0 \in \mathcal{N}_1,~
\exists i_r \in \mathcal{N}_2,~
\exists i_1, \dots, i_{r-1} \in \mathcal{N} \quad
\text{s.t.}\quad
a_{i_{k-1}i_{k}} \neq 0, \quad
k = 1, \dots, r.
\end{align}
If perceiving $\vecc{A}$ as a directed graph adjacency matrix of vertices labeled by $\mathcal{N}$, then \eqref{eq:connect} simply means that there exists a directed path from any vertex in $\mathcal{N}_1$ to at least one vertex in $\mathcal{N}_2$. In particular, if $\mathcal{N}_1=\emptyset$, then any matrix $\vecc{A}$ connects $\mathcal{N}_1$ and $\mathcal{N}_2$.
\end{defn}
\begin{defn}
Given a square matrix $\vecc{A}$ and a column vector $\vec{x}$, define
\begin{align*}
\mathcal{N}^0(\vecc{A}\vec{x}) = \{i:~(\vecc{A}\vec{x})_i = 0\}
\quad\text{and}\quad
\mathcal{N}^+(\vecc{A}\vec{x}) = \{i:~(\vecc{A}\vec{x})_i > 0\}.
\end{align*}
\end{defn}
The following theorem was proved in \cite{lorenz1977inversmonotonie}, see also \cite{li2020monotonicity} for a detailed proof. 
\begin{thm}[Lorenz's condition]\label{thm:Lorenz_cond}
If $\vecc{A}_a^-$ has a decomposition $\vecc{A}_a^- = \vecc{A}^z + \vecc{A}^s = (a^z_{ij}) + (a^s_{ij})$ with $\vecc{A}^z \leq 0$ and $\vecc{A}^s\leq 0$, such that
\begin{itemize}
\item[1.] $\vecc{A}_d + \vecc{A}^z$ is a nonsingular M-matrix,
\item[2.] $\vecc{A}_a^+ \leq \vecc{A}^z\vecc{A}_d^{-1}\vecc{A}^s$ or equivalently $\forall a_{ij} > 0$ with $i\neq j$, $a_{ij} \leq \displaystyle\sum_{k=1}^n a^z_{ik} a_{kk}^{-1} a^s_{kj}$,
\item[3.] $\exists\vec{e}\in\bR^n\setminus\{\vec{0}\}$, $\vec{e}\geq 0$ with $\vecc{A}\vec{e}\geq 0$ such that $\vecc{A}^z$ or $\vecc{A}^s$ connects $\mathcal{N}^0(\vecc{A}\vec{e})$ with $\mathcal{N}^+(\vecc{A}\vec{e})$.
\end{itemize}
Then $\vecc{A}$ is a product of two nonsingular M-matrices, thus $\vecc{A}^{-1}\geq0$.
\end{thm}
In the rest of this section, to obtain monotonicity, we will show that  the fourth order scheme matrix satisfies the conditions in Theorem \ref{thm:Lorenz_cond} under suitable mesh and time step constraints.

\subsection{The fourth order scheme in one dimension}\label{sec:1D4order} 
In general, the high order finite element methods do not have an M-matrix structure. But it is possible to show that they are products of M-matrices. 
Next we verify that  the Lorenz's condition in Theorem~\ref{thm:Lorenz_cond} can be satisfied for the matrix in the scheme \eqref{1d-4th-scheme}.  
For convenience of writing and similar to references \cite{hu2021positivity,shen2021discrete}, we use operator notation.  Let $\mathcal{A}$ be  the linear operator  corresponding the scheme matrix $\vecc{A}$. The linear operator $\mathcal{A}_d$ (associated with the diagonal matrix $\vecc{A}_d$) is:
\begin{align*}
&\mbox{If $i$ is odd,}\quad
\mathcal{A}_d(\vec{g}^{\,n+1})_{i}
=\, \invm_{i}g^{n+1}_{i}
+ \Delta{t}\,\frac{D(\invm_{i-2} + 4\invm_{i-1} + 18\invm_{i} + 4\invm_{i+1} + \invm_{i+2})}{8h^2}g^{n+1}_{i};&\\
&\mbox{if $i$ is even,}\quad
\mathcal{A}_d(\vec{g}^{\,n+1})_{i}
= \invm_{i}g^{n+1}_{i}
+ \Delta{t}\, \frac{D(\invm_{i-1}+\invm_{i+1})}{h^2}g^{n+1}_{i}.
\end{align*}
Let $a^+ = \max\{a,0\}$ be the positive part and $a^- = -\min\{a,0\}$ negative parts of a number $a$. The operator $\mathcal{A}_a^+$ (associated with the matrix $\vecc{A}_a^+$) is given by:
\begin{align*}
&\mbox{If $i$ is odd,}\\
&\resizebox{.99\hsize}{!}{$\mathcal{A}_a^+(\vec{g}^{\,n+1})_{i}
= \left(-\frac{\Delta{t}}{h}\frac{u_{i-2}}{4} + \Delta{t}\,\frac{D(3\invm_{i-2} - 4\invm_{i-1} + 3\invm_{i})}{8h^2}\right)^+ g^{n+1}_{i-2}
+ \left(\frac{\Delta{t}}{h}\frac{u_{i+2}}{4} + \Delta{t}\,\frac{D(3\invm_{i} - 4\invm_{i+1} + 3\invm_{i+2})}{8h^2}\right)^+ g^{n+1}_{i+2}$};\\
&\mbox{if $i$ is even,}\quad
\mathcal{A}_a^+(\vec{g}^{\,n+1})_{i} = 0.
\end{align*}
By definition of $(\cdot)^+$, it is straightforward to see the matrix $\vecc{A}^{+}_a$ is entry-wise non-negative. Let $\vecc{A}_a^- = \vecc{A} - \vecc{A}_d - \vecc{A}_a^+$ and we further split it by introducing the operator $\mathcal{A}^z$ (associated with the matrix $\vecc{A}^z$) as follows:
\begin{align*}
&\mbox{If $i$ is odd,}\\
&\resizebox{.99\hsize}{!}{$\mathcal{A}^{z}(\vec{g}^{\,n+1})_{i} =
- \left(-\frac{\Delta{t}}{h}u_{i-1} + \Delta{t}\,\frac{D(4\invm_{i-2} + 12\invm_{i})}{8h^2} - \Big(-\frac{\Delta{t}}{h}\frac{u_{i-2}}{4} + \Delta{t}\,\frac{D(3\invm_{i-2} - 4\invm_{i-1} + 3\invm_{i})}{8h^2}\Big)^+ \right) g^{n+1}_{i-1}$} \\
&\resizebox{.99\hsize}{!}{$- \left(\frac{\Delta{t}}{h}u_{i+1} + \Delta{t}\,\frac{D(12\invm_{i} + 4\invm_{i+2})}{8h^2} - \Big(\frac{\Delta{t}}{h}\frac{u_{i+2}}{4} + \Delta{t}\,\frac{D(3\invm_{i} - 4\invm_{i+1} + 3\invm_{i+2})}{8h^2}\Big)^+ \right) g^{n+1}_{i+1}$} \\
&\resizebox{.99\hsize}{!}{$- \left(-\frac{\Delta{t}}{h}\frac{u_{i-2}}{4} + \Delta{t}\,\frac{D(3\invm_{i-2} - 4\invm_{i-1} + 3\invm_{i})}{8h^2}\right)^- g^{n+1}_{i-2}
- \left(\frac{\Delta{t}}{h}\frac{u_{i+2}}{4} + \Delta{t}\,\frac{D(3\invm_{i} - 4\invm_{i+1} + 3\invm_{i+2})}{8h^2}\right)^- g^{n+1}_{i+2}$}; \\
&\mbox{if $i$ is even,}\quad
\mathcal{A}^{z}(\vec{g}^{\,n+1})_{i} = 0.
\end{align*}
Let $\vecc{A}^s = \vecc{A}_a^- -\vecc{A}^z$. The operator $\mathcal{A}^s$ (associated with the matrix $\vecc{A}^s$) is as follows:
\begin{align*}
&\mbox{If $i$ is odd,}\\
&\resizebox{.99\hsize}{!}{$\mathcal{A}^s(\vec{g}^{\,n+1})_{i}
= -\left(-\frac{\Delta{t}}{h}\frac{u_{i-2}}{4} + \Delta{t}\,\frac{D(3\invm_{i-2} - 4\invm_{i-1} + 3\invm_{i})}{8h^2}\right)^+ g^{n+1}_{i-2}
- \left(\frac{\Delta{t}}{h}\frac{u_{i+2}}{4} + \Delta{t}\,\frac{D(3\invm_{i} - 4\invm_{i+1} + 3\invm_{i+2})}{8h^2}\right)^+ g^{n+1}_{i+2}$};\\
&\mbox{if $i$ is even,}\\
&\resizebox{.99\hsize}{!}{$\mathcal{A}^{s}(\vec{g}^{\,n+1})_{i} = 
-\left(-\frac{\Delta{t}}{h}\frac{u_{i-1}}{2} + \frac{\Delta{t}}{h^2}\frac{D(3\invm_{i-1} + \invm_{i+1})}{4}\right)g^{n+1}_{i-1}
-\left(\frac{\Delta{t}}{h}\frac{u_{i+1}}{2} + \frac{\Delta{t}}{h^2}\frac{D(\invm_{i-1} + 3\invm_{i+1})}{4}\right)g^{n+1}_{i+1}$}.
\end{align*}
It is easy to verify that the following mesh constraint is sufficient for $\vecc{A}^z\leq0$ and $\vecc{A}^s\leq0$:
\begin{align}\label{eq:1D4order_suff1}
h\max\{|u_i|,|u_{i+1}|,|u_{i+2}|\}\leq D\min\{\invm_i,\invm_{i+1},\invm_{i+2}\},\quad
\text{for~odd}~i.
\end{align}
The matrix $\vecc{A}_d+\vecc{A}^z$ is an $N$-by-$N$ real square matrix with positive diagonal entries and nonpositive off-diagonals. We use Lemma~\ref{thm:M_matrix1} to verify the first condition in Theorem~\ref{thm:Lorenz_cond}. Let $\vecc{D}$ equal to the identity matrix and let $\vec{1} = [1,\cdots,1]^\mathrm{T}$. Then, the row sum of $(\vecc{A}_d+\vecc{A}^z)\vecc{D}$ can be evaluated by $(\vecc{A}_d+\vecc{A}^z)\vec{1}$, namely
\begin{align*}
&\mbox{If $i$ is odd, the sum of the i-th row is:}\quad 
\invm_{i} + \frac{\Delta{t}}{4h}(-u_{i-2}+4u_{i-1}-4u_{i+1}+u_{i+2});\\
&\mbox{if $i$ is even, the sum of the i-th row is:}\quad
\invm_{i} + \frac{\Delta{t}}{h^2}D(\invm_{i-1} + \invm_{i+1}).
\end{align*}
To guarantee the positive row sums of the matrix $(\vecc{A}_d+\vecc{A}^z)\vecc{D}$, the following constraints on time step size are sufficient:
\begin{subequations}\label{eq:1D4order_suff2}
\begin{align}\label{eq:1D4order_suff2a}
\invm_{i} + \Delta{t}\frac{-u_{i-2}+4u_{i-1}-4u_{i+1}+u_{i+2}}{4h} > 0, \quad\text{for odd}~i.
\end{align}
Recall that the velocity field $u(x)$ is incompressible in Model~1, thus $u \equiv C$ in one dimension. So \eqref{eq:1D4order_suff2} is trivially satisfied for positive measure $\invm$.
Thus $\vecc{A}_d+\vecc{A}^z$ is a nonsingular M-matrix. 
Meanwhile, the divergence free velocity in one dimension also implies $\frac{u_{i-1}-u_{i+1}}{2h} = 0$, namely 
\begin{align}\label{eq:1D4order_suff2b}
\invm_{i} + \Delta{t}\frac{u_{i-1}-u_{i+1}}{2h} > 0, \quad\text{for even}~i.
\end{align}
\end{subequations}
Thus, we have $\vecc{A}\vec{1} > 0$.
Therefore, $\mathcal{N}^0(\vecc{A}\vec{1}) = \emptyset$ and the third condition in Theorem~\ref{thm:Lorenz_cond} is trivially satisfied. Our next goal is to seek a sufficient condition such that the second condition in Theorem~\ref{thm:Lorenz_cond} hold. 
By comparing $\mathcal A_a^+(\vec g^{n+1})_i$ with $\mathcal A^z\circ \mathcal (A^d)^{-1} \circ\mathcal A^s(\vec g^{n+1})_i$, it is straightforward to verify that
$\vecc{A}_a^+\leq \vecc{A}^z\vecc{A}_d^{-1}\vecc{A}^s$ is equivalent to the following: for odd $i$,
\begin{align}\notag
&\left(\invm_{i-1} - \frac{\Delta{t}}{h}\frac{u_{i-2}}{2} + \frac{\Delta{t}}{h^2}\frac{D(7\invm_{i-2} + 5\invm_{i})}{4}\right)
\left(-\frac{\Delta{t}}{h}\frac{u_{i-2}}{4} + \frac{\Delta{t}}{h^2}\frac{D(3\invm_{i-2}-4\invm_{i-1}+3\invm_{i})}{8}\right)\\
\leq&\,
\left(- \frac{\Delta{t}}{h}\frac{u_{i-2}}{2} + \frac{\Delta{t}}{h^2}\frac{D(3\invm_{i-2} + \invm_{i})}{4}\right)
\left(- \frac{\Delta{t}}{h}u_{i-1} + \frac{\Delta{t}}{h^2}\frac{D(4\invm_{i-2} + 12\invm_{i})}{8}\right).\label{add-condition1}
\end{align}
Multiply $32\,(\frac{h^2}{\Delta{t}})^2$ on both side of above inequality, after some manipulation, we get:
\begin{align*}
&\Big(4\frac{h^2}{\Delta{t}}\invm_{i-1} - 2h u_{i-2} + D(7\invm_{i-2} + 5\invm_{i})\Big)
\Big(-2h u_{i-2} + D(3\invm_{i-2}-4\invm_{i-1}+3\invm_{i})\Big)\\
\leq&\,
\Big(- 2h u_{i-2} + D(3\invm_{i-2} + \invm_{i})\Big)
\Big(-8h u_{i-1} + D(4\invm_{i-2} + 12\invm_{i})\Big).
\end{align*}
Let $b = \max\{\invm_{i-2},\invm_{i-1},\invm_{i}\}$ and $s = \min\{\invm_{i-2},\invm_{i-1},\invm_{i}\}$, namely, the largest and smallest quadrature point values of $\invm$ on an element $[x_{i-2},x_i]$. Assume the finite difference grid spacing satisfies:
\begin{align}\label{eq:1D4order_suff4}
h\max\{|u_{i-2}|, |u_{i-1}|, |u_i|\} \leq \frac{1}{4}D\min\{\invm_{i-2}, \invm_{i-1}, \invm_i\},\quad\text{for odd}~i.
\end{align}
Note, \eqref{eq:1D4order_suff4} implies \eqref{eq:1D4order_suff1}. It is easy to verify that a sufficient condition for \eqref{add-condition1} is
\begin{align*}
\Big(\frac{1}{4}Ds + 2(3D+\frac{h^2}{\Delta{t}})b\Big)
\Big(\frac{1}{2}Ds + D(6b-4s)\Big) \leq
(2Ds-\frac{1}{4}Ds) (16Ds-2Ds).
\end{align*}
Therefore, a sufficient condition is:
\begin{align*}
3D+\frac{h^2}{\Delta{t}} \leq \frac{49D s^2}{2b(12b-7s)} - \frac{D s}{8b}.
\end{align*}
Now, we simplify the sufficient condition above. 
The invariant measure $\invm\geq \epsilon_0>0$, define $r = b/s$, then above inequality can be rewritten as  
\begin{align*}
\frac{h^2}{\Delta{t}} 
\leq \frac{49D}{2r(12r-7)} - \frac{D}{8r} - 3D
= \frac{7}{2}D\Big(\frac{1}{r-\frac{7}{12}} - \frac{1}{r}\Big) - \frac{D}{8r} - 3D.
\end{align*}
From the definition of $r$, we know $r \geq 1$. Thus, it is sufficient to employ the conditions $r\in[1,1.15]$ and
\begin{align*}
\frac{h^2}{\Delta{t}} \leq 0.02D
< \min_{r\in[1,1.15]} \left\{\frac{7}{2}D\Big(\frac{1}{r-\frac{7}{12}} - \frac{1}{r}\Big) - \frac{D}{8r} - 3D\right\}.
\end{align*}
This indicates we only need to find a suitable upper bound on $h$ such that $b \leq 1.15s$ (namely $r\in[1,1.15]$) holds. Recall $\invm$ is continuously differentiable. Assume $\invm$ take its maximum at point $x^\ast$ on cell $[x_{i-2},x_i]$ and $\invm$ take its minimum at point $x_\ast$ on cell $[x_{i-2},x_i]$. 
By mean value theorem, there exist a point $\xi \in [x_{i-2},x_i]$ such that
\begin{align*}
\invm(x^\ast) = \invm(x_\ast) + (x^\ast - x_\ast)\invm'{(\xi)}.
\end{align*}
Therefore
\begin{align*}
b \leq \invm(x^\ast) = \invm(x_\ast) + (x^\ast - x_\ast)\invm'{(\xi)}
\leq s + 2h \max_{[x_{i-2},x_i]}|\invm'|,
\end{align*}
which means in order to let $b \leq 1.15s$ hold, we can employ a sufficient condition as follows
\begin{align*}
s + 2h \max_{[x_{i-2},x_i]}|\invm'| \leq 1.15s.
\end{align*}
To this end, we obtain a constraint on $h$, as follows
\begin{align}\label{eq:1D4order_suff5}
h \max_{[x_{i-2},x_i]}|\invm'| \leq 0.075\min\{\invm_{i-2},\invm_{i-1},\invm_i\},\quad\text{for odd}~i.
\end{align}
As a summary, we have the following theorem:
\begin{thm}\label{thm:1D4order_suff_conds}
Under the mesh and time step constraints \eqref{eq:1D4order_suff2}, \eqref{eq:1D4order_suff4}, \eqref{eq:1D4order_suff5} and $\frac{\Delta t}{h^2}\geq \frac{50}{D}$, the coefficient matrix for the unknown vector $\vec{g}^{\,n+1}$ in the fourth order finite difference scheme \eqref{1d-4th-scheme} satisfies the Lorenz's conditions, so it is a product of two M-matrices thus  monotone. In particular, in one dimension, for a discrete divergence free velocity field $u$ (which is constant), the matrix in fourth order finite difference scheme \eqref{1d-4th-scheme} is monotone under the following constraints: 
 $\frac{\Delta t}{h^2}\geq \frac{50}{D}$ and, for odd $i$,
\begin{align*}
h\max\{|u_i|,|u_{i+1}|,|u_{i+2}|\}\leq\frac{1}{4}D\min\{\invm_i,\invm_{i+1},\invm_{i+2}\},\\
h \max_{[x_{i},x_{i+2}]}|\invm'| \leq 0.075\min\{\invm_{i},\invm_{i+1},\invm_{i+2}\}.
\end{align*} 
\end{thm} 
\begin{rem}
 In practice, to realize the mesh size and time step satifying the constraints above, one can first choose a small enough mesh size $h$, then choose a large enough time step $\Delta t$. For instance, for a constant velocity case, for a small enough $h$, one can use $\Delta t\geq \frac{50}{D} h^2$. We emphasize that the sufficient conditions above are not sharp for monotonicity to hold, but with a fixed mesh size $h$ monotonicity will be lost in the fourth order scheme when $\Delta t\to 0$.  
\end{rem}

\subsection{The fourth order scheme in two dimension}
Next we  derive a sufficient mesh size and time step conditions
for the two-dimensional fourth order scheme to satisfy the Lorenz's conditions in Theorem~\ref{thm:Lorenz_cond}. For convenience, we follow \cite{hu2021positivity,shen2021discrete}, to use operator notation for all matrices.

Similar to the one dimensional discussion above, $\mathcal{A}$ denotes the linear operator for the scheme matrix. The linear operator $\mathcal{A}_d$ (associated with the diagonal matrix $\vecc{A}_d$) is:
\begin{align*}
\resizebox{.99\hsize}{!}{$\mbox{If $(x_i,y_j)$ is a knot,}\quad
\mathcal{A}_d(\vec{g}^{\,n+1})_{ij}
=\, \invm_{i,j}g^{n+1}_{ij}
+ \Delta{t}\,\frac{D(\invm_{i-2,j} + 4\invm_{i-1,j} + 18\invm_{i,j} + 4\invm_{i+1,j} + \invm_{i+2,j})}{8h^2}g^{n+1}_{ij}$}&\\
+\, \Delta{t}\,\frac{D(\invm_{i,j-2} + 4\invm_{i,j-1} + 18\invm_{i,j} + 4\invm_{i,j+1} + \invm_{i,j+2})}{8h^2}g^{n+1}_{ij};&\\
\resizebox{.99\hsize}{!}{$\mbox{if $(x_i,y_j)$ is an edge (parallel to $y$-axis) center,}\quad
\mathcal{A}_d(\vec{g}^{\,n+1})_{ij}
=\, \invm_{i,j}g^{n+1}_{ij}
+ \Delta{t}\, \frac{D(\invm_{i,j-1} + \invm_{i,j+1})}{h^2}g^{n+1}_{ij}$}&\\
+\, \Delta{t}\, \frac{D(\invm_{i-2,j} + 4\invm_{i-1,j} + 18\invm_{ij} + 4\invm_{i+1,j} + \invm_{i+2,j})}{8h^2}g^{n+1}_{ij};&\\
\resizebox{.99\hsize}{!}{$\mbox{if $(x_i,y_j)$ is a cell center,}\quad
\mathcal{A}_d(\vec{g}^{\,n+1})_{ij}
= \invm_{i,j}g^{n+1}_{ij}
+ \Delta{t}\, \frac{D(\invm_{i-1,j}+\invm_{i+1,j}+\invm_{i,j-1}+\invm_{i,j+1})}{h^2}g^{n+1}_{ij}.$}&
\end{align*}
For $(x_i,y_j)$ is an edge (parallel to $x$-axis) center, this case is very similar to the case which $(x_i,y_j)$ is an edge (parallel to $y$-axis) center, thus omitted.
For the sake of brevity, we omit the case $(x_i,y_j)$ is an edge (parallel to $x$-axis) center when defining operators.
Recall that we use notation $(\cdot)^+ = \max\{\cdot,0\}$ to denote the positive part and $(\cdot)^- = -\min\{\cdot,0\}$ to denote the negative parts of a number $(\cdot)$.
The operator $\mathcal{A}_a^+$ (associated with the matrix $\vecc{A}_a^+$) is given by:
\begin{align*}
&\mbox{If $(x_i,y_j)$ is a knot,}\\
&\resizebox{.99\hsize}{!}{$\mathcal{A}_a^+(\vec{g}^{\,n+1})_{ij}
= \left(-\frac{\Delta{t}}{h}\frac{u_{i-2,j}}{4} + \Delta{t}\,\frac{D(3\invm_{i-2,j} - 4\invm_{i-1,j} + 3\invm_{i,j})}{8h^2}\right)^+ g^{n+1}_{i-2,j}
+ \left(\frac{\Delta{t}}{h}\frac{u_{i+2,j}}{4} + \Delta{t}\,\frac{D(3\invm_{i,j} - 4\invm_{i+1,j} + 3\invm_{i+2,j})}{8h^2}\right)^+ g^{n+1}_{i+2,j}$}\\
&\resizebox{.99\hsize}{!}{$+ \left(-\frac{\Delta{t}}{h}\frac{v_{i,j-2}}{4} + \Delta{t}\,\frac{D(3\invm_{i,j-2} - 4\invm_{i,j-1} + 3\invm_{i,j})}{8h^2}\right)^+ g^{n+1}_{i,j-2}
+ \left(\frac{\Delta{t}}{h}\frac{v_{i,j+2}}{4} + \Delta{t}\,\frac{D(3\invm_{i,j} - 4\invm_{i,j+1} + 3\invm_{i,j+2})}{8h^2}\right)^+ g^{n+1}_{i,j+2}$};\\
&\mbox{if $(x_i,y_j)$ is an edge (parallel to $y$-axis) center,}\\
&\resizebox{.99\hsize}{!}{$\mathcal{A}_a^+(\vec{g}^{\,n+1})_{ij} 
= \left(-\frac{\Delta{t}}{h}\frac{u_{i-2,j}}{4} + \Delta{t}\,\frac{D(3\invm_{i-2,j} - 4\invm_{i-1,j} + 3\invm_{i,j})}{8h^2}\right)^+ g^{n+1}_{i-2,j}
+ \left(\frac{\Delta{t}}{h}\frac{u_{i+2,j}}{4} + \Delta{t}\,\frac{D(3\invm_{i+2,j} - 4\invm_{i+1,j} + 3\invm_{i,j})}{8h^2}\right)^+ g^{n+1}_{i+2,j}$};\\
&\mbox{if $(x_i,y_j)$ is a cell center,}\quad
\mathcal{A}_a^+(\vec{g}^{\,n+1})_{ij} = 0.
\end{align*}
It is straightforward to see the matrix $\vecc{A}^{+}_a$ is entry-wise non-negative. Let $\vecc{A}_a^- = \vecc{A} - \vecc{A}_d - \vecc{A}_a^+$ and we further split it by introducing the operator $\mathcal{A}^z$ (associated with the matrix $\vecc{A}^z$) as follows:
\begin{align*}
&\mbox{If $(x_i,y_j)$ is a knot,}\\
&\resizebox{.99\hsize}{!}{$\mathcal{A}^{z}(\vec{g}^{\,n+1})_{ij} =
- \left(-\frac{\Delta{t}}{h}u_{i-1,j} + \Delta{t}\,\frac{D(4\invm_{i-2,j} + 12\invm_{i,j})}{8h^2} - \Big(-\frac{\Delta{t}}{h}\frac{u_{i-2,j}}{4} + \Delta{t}\,\frac{D(3\invm_{i-2,j} - 4\invm_{i-1,j} + 3\invm_{i,j})}{8h^2}\Big)^+ \right) g^{n+1}_{i-1,j}$} \\
&\resizebox{.99\hsize}{!}{$- \left(\frac{\Delta{t}}{h}u_{i+1,j} + \Delta{t}\,\frac{D(12\invm_{i,j} + 4\invm_{i+2,j})}{8h^2} - \Big(\frac{\Delta{t}}{h}\frac{u_{i+2,j}}{4} + \Delta{t}\,\frac{D(3\invm_{i,j} - 4\invm_{i+1,j} + 3\invm_{i+2,j})}{8h^2}\Big)^+ \right) g^{n+1}_{i+1,j}$} \\
&\resizebox{.99\hsize}{!}{$- \left(-\frac{\Delta{t}}{h}v_{i,j-1} + \Delta{t}\,\frac{D(4\invm_{i,j-2} + 12\invm_{i,j})}{8h^2} - \Big(-\frac{\Delta{t}}{h}\frac{v_{i,j-2}}{4} + \Delta{t}\,\frac{D(3\invm_{i,j-2} - 4\invm_{i,j-1} + 3\invm_{i,j})}{8h^2}\Big)^+ \right) g^{n+1}_{i,j-1}$} \\
&\resizebox{.99\hsize}{!}{$- \left(\frac{\Delta{t}}{h}v_{i,j+1} + \Delta{t}\,\frac{D(12\invm_{i,j} + 4\invm_{i,j+2})}{8h^2} - \Big(\frac{\Delta{t}}{h}\frac{v_{i,j+2}}{4} + \Delta{t}\,\frac{D(3\invm_{i,j} - 4\invm_{i,j+1} + 3\invm_{i,j+2})}{8h^2}\Big)^+ \right) g^{n+1}_{i,j+1}$} \\ 
&\resizebox{.99\hsize}{!}{$- \left(-\frac{\Delta{t}}{h}\frac{u_{i-2,j}}{4} + \Delta{t}\,\frac{D(3\invm_{i-2,j} - 4\invm_{i-1,j} + 3\invm_{i,j})}{8h^2}\right)^- g^{n+1}_{i-2,j}
- \left(\frac{\Delta{t}}{h}\frac{u_{i+2,j}}{4} + \Delta{t}\,\frac{D(3\invm_{i,j} - 4\invm_{i+1,j} + 3\invm_{i+2,j})}{8h^2}\right)^- g^{n+1}_{i+2,j}$} \\ 
&\resizebox{.99\hsize}{!}{$- \left(-\frac{\Delta{t}}{h}\frac{v_{i,j-2}}{4} + \Delta{t}\,\frac{D(3\invm_{i,j-2} - 4\invm_{i,j-1} + 3\invm_{i,j})}{8h^2}\right)^- g^{n+1}_{i,j-2}
- \left(\frac{\Delta{t}}{h}\frac{v_{i,j+2}}{4} + \Delta{t}\,\frac{D(3\invm_{i,j} - 4\invm_{i,j+1} + 3\invm_{i,j+2})}{8h^2}\right)^- g^{n+1}_{i,j+2}$};\\
&\mbox{if $(x_i,y_j)$ is an edge (parallel to $y$-axis) center,}\\
&\resizebox{.99\hsize}{!}{$\mathcal{A}^{z}(\vec{g}^{\,n+1})_{ij} = 
- \left(-\frac{\Delta{t}}{h}u_{i-1,j} + \Delta{t}\,\frac{D(4\invm_{i-2,j} + 12\invm_{i,j})}{8h^2} - \Big(-\frac{\Delta{t}}{h}\frac{u_{i-2,j}}{4} + \Delta{t}\,\frac{D(3\invm_{i-2,j} - 4\invm_{i-1,j} + 3\invm_{i,j})}{8h^2}\Big)^+\right) g^{n+1}_{i-1,j}$} \\
&\resizebox{.99\hsize}{!}{$- \left(\frac{\Delta{t}}{h}u_{i+1,j} + \Delta{t}\,\frac{D(12\invm_{i,j} + 4\invm_{i+2,j})}{8h^2} - \Big(\frac{\Delta{t}}{h}\frac{u_{i+2,j}}{4} + \Delta{t}\,\frac{D(3\invm_{i,j} - 4\invm_{i+1,j} + 3\invm_{i+2,j})}{8h^2}\Big)^+\right) g^{n+1}_{i+1,j}$} \\
&\resizebox{.99\hsize}{!}{$- \left(-\frac{\Delta{t}}{h}\frac{u_{i-2,j}}{4} + \Delta{t}\,\frac{D(3\invm_{i-2,j} - 4\invm_{i-1,j} + 3\invm_{i,j})}{8h^2}\right)^- g^{n+1}_{i-2,j}
- \left(\frac{\Delta{t}}{h}\frac{u_{i+2,j}}{4} + \Delta{t}\,\frac{D(3\invm_{i+2,j} - 4\invm_{i+1,j} + 3\invm_{i,j})}{8h^2}\right)^- g^{n+1}_{i+2,j}$};\\
&\mbox{If $(x_i,y_j)$ is a cell center,}\quad
\mathcal{A}^{z}(\vec{g}^{\,n+1})_{ij} = 0.
\end{align*}
Similar to \eqref{eq:1D4order_suff1} in subsection~\ref{sec:1D4order}, it is easy to verify that $\vecc{A}^z \leq 0$ under the following sufficient condition: for odd $i$ and odd $j$,
\begin{subequations}\label{eq:2D4order_suff1}
\begin{align}
h\max\{|u_{i,j}|,|u_{i+1,j}|,|u_{i+2,j}|\} &\leq D\min\{\invm_{i,j},\invm_{i+1,j},\invm_{i+2,j}\},\\
\text{and}\quad
h\max\{|v_{i,j}|,|v_{i,j+1}|,|v_{i,j+2}|\} &\leq D\min\{\invm_{i,j},\invm_{i,j+1},\invm_{i,j+2}\}.
\end{align}
\end{subequations}
The matrix $\vecc{A}^s = \vecc{A}_a^{-} - \vecc{A}^z$. Therefore, the operator $\mathcal{A}^s$ (associated with the matrix $\vecc{A}^s$) is as follows:
\begin{align*}
&\mbox{If $(x_i,y_j)$ is a knot,}\\
&\resizebox{.99\hsize}{!}{$\mathcal{A}^{s}(\vec{g}^{\,n+1})_{ij} =
- \left(-\frac{\Delta{t}}{h}\frac{u_{i-2,j}}{4} + \Delta{t}\,\frac{D(3\invm_{i-2,j} - 4\invm_{i-1,j} + 3\invm_{i,j})}{8h^2}\right)^+ g^{n+1}_{i-1,j}
- \left(\frac{\Delta{t}}{h}\frac{u_{i+2,j}}{4} + \Delta{t}\,\frac{D(3\invm_{i,j} - 4\invm_{i+1,j} + 3\invm_{i+2,j})}{8h^2}\right)^+ g^{n+1}_{i+1,j}$} \\
&\resizebox{.99\hsize}{!}{$- \left(-\frac{\Delta{t}}{h}\frac{v_{i,j-2}}{4} + \Delta{t}\,\frac{D(3\invm_{i,j-2} - 4\invm_{i,j-1} + 3\invm_{i,j})}{8h^2}\right)^+ g^{n+1}_{i,j-1}
- \left(\frac{\Delta{t}}{h}\frac{v_{i,j+2}}{4} + \Delta{t}\,\frac{D(3\invm_{i,j} - 4\invm_{i,j+1} + 3\invm_{i,j+2})}{8h^2}\right)^+ g^{n+1}_{i,j+1}$};\\
&\mbox{if $(x_i,y_j)$ is an edge (parallel to $y$-axis) center,}\\
&\resizebox{.99\hsize}{!}{$\mathcal{A}^{s}(\vec{g}^{\,n+1})_{ij} =
- \left(-\frac{\Delta{t}}{h}\frac{u_{i-2,j}}{4} + \Delta{t}\,\frac{D(3\invm_{i-2,j} - 4\invm_{i-1,j} + 3\invm_{i,j})}{8h^2}\right)^+ g^{n+1}_{i-1,j}
- \left(\frac{\Delta{t}}{h}\frac{u_{i+2,j}}{4} + \Delta{t}\,\frac{D(3\invm_{i,j} - 4\invm_{i+1,j} + 3\invm_{i+2,j})}{8h^2}\right)^+ g^{n+1}_{i+1,j}$}\\
&\resizebox{.99\hsize}{!}{$- \left(-\frac{\Delta{t}}{h}\frac{v_{i,j-1}}{2} + \Delta{t}\, \frac{D(3\invm_{i,j-1} + \invm_{i,j+1})}{4h^2}\right) g^{n+1}_{i,j-1}
- \left(\frac{\Delta{t}}{h}\frac{v_{i,j+1}}{2} + \Delta{t}\, \frac{D(\invm_{i,j-1} + 3\invm_{i,j+1})}{4h^2}\right) g^{n+1}_{i,j+1}$};\\
&\mbox{if $(x_i,y_j)$ is a cell center,}\\
&\resizebox{.99\hsize}{!}{$\mathcal{A}^{s}(\vec{g}^{\,n+1})_{ij} =
- \left(-\frac{\Delta{t}}{h}\frac{u_{i-1,j}}{2} + \Delta{t}\, \frac{D(3\invm_{i-1,j} + \invm_{i+1,j})}{4h^2}\right) g^{n+1}_{i-1,j}
- \left(\frac{\Delta{t}}{h}\frac{u_{i+1,j}}{2} + \Delta{t}\, \frac{D(\invm_{i-1,j} + 3\invm_{i+1,j})}{4h^2}\right) g^{n+1}_{i+1,j}$} \\
&\resizebox{.99\hsize}{!}{$- \left(-\frac{\Delta{t}}{h}\frac{v_{i,j-1}}{2} + \Delta{t}\, \frac{D(3\invm_{i,j-1} + \invm_{i,j+1})}{4h^2}\right) g^{n+1}_{i,j-1} 
- \left(\frac{\Delta{t}}{h}\frac{v_{i,j+1}}{2} + \Delta{t}\, \frac{D(\invm_{i,j-1} + 3\invm_{i,j+1})}{4h^2}\right) g^{n+1}_{i,j+1}$}.
\end{align*}
Obviously, under the sufficient condition \eqref{eq:2D4order_suff1}, $\vecc{A}^s\leq0$ also holds. The matrix $\vecc{A}_d + \vecc{A}^z$ is a real squared matrix with positive diagonal entries and nonpositive off-diagonals. We use Lemma~\ref{thm:M_matrix1} to verify the first condition in Theorem~\ref{thm:Lorenz_cond}. Let $\vecc{D}$ equal to the identity matrix and applying the same argument as in subsection~\ref{sec:1D4order}. Notice that the row sums of the matrix $\vecc{A}_d + \vecc{A}^z$ are the outputs of $[\mathcal A_d+\mathcal A^z](\vec 1)_{ij}$:
\begin{align*}
&\resizebox{.99\hsize}{!}{$\mbox{If $(x_i,y_j)$ is a knot,}\quad
\invm_{i,j}
+ \Delta{t}\frac{-u_{i-2,j} + 4u_{i-1,j} - 4u_{i+1,j} + u_{i+2,j}}{4h}
+ \Delta{t}\frac{-v_{i,j-2} + 4v_{i,j-1} - 4v_{i,j+1} + v_{i,j+2}}{4h} $};\\
&\resizebox{.99\hsize}{!}{$\mbox{if $(x_i,y_j)$ is an edge (parallel to $y$-axis) center,}\quad
\invm_{i,j}
+ \Delta{t}\, \frac{D(\invm_{i,j-1} + \invm_{i,j+1})}{h^2}
+ \Delta{t}\frac{- u_{i-2,j} + 4u_{i-1,j} - 4u_{i+1,j} + u_{i+2,j}}{4h} $};\\ 
&\resizebox{.99\hsize}{!}{$\mbox{if $(x_i,y_j)$ is an edge (parallel to $x$-axis) center,}\quad
\invm_{i,j}
+ \Delta{t}\, \frac{D(\invm_{i-1,j} + \invm_{i+1,j})}{h^2}
+ \Delta{t}\frac{- v_{i,j-2} + 4v_{i,j-1} - 4v_{i,j+1} + v_{i,j+2}}{4h} $};\\
&\mbox{if $(x_i,y_j)$ is a cell center,}\quad
\invm_{i,j}
+ \Delta{t}\, \frac{D(\invm_{i-1,j}+\invm_{i+1,j}+\invm_{i,j-1}+\invm_{i,j+1})}{h^2}.
\end{align*}
To guarantee the positive row sums, the following constraints on time step size are sufficient:
\begin{subequations}\label{eq:2D4order_suff2}
\begin{align}
&\mbox{For odd $i$ and odd $j$,}\nonumber\\
&\invm_{i,j}
\!+\! \Delta{t}\frac{-u_{i-2,j} + 4u_{i-1,j} - 4u_{i+1,j} + u_{i+2,j}}{4h}
\!+\! \Delta{t}\frac{-v_{i,j-2} + 4v_{i,j-1} - 4v_{i,j+1} + v_{i,j+2}}{4h} \! >\! 0;\\
&\mbox{for odd $i$ and even $j$,}\nonumber\\
&\invm_{i,j}
+ \Delta{t}\, \frac{D(\invm_{i,j-1} + \invm_{i,j+1})}{h^2}
+ \Delta{t}\frac{- u_{i-2,j} + 4u_{i-1,j} - 4u_{i+1,j} + u_{i+2,j}}{4h} > 0;\\
&\mbox{for even $i$ and odd $j$,}\nonumber\\
&\invm_{i,j}
+ \Delta{t}\, \frac{D(\invm_{i-1,j} + \invm_{i+1,j})}{h^2}
+ \Delta{t}\frac{- v_{i,j-2} + 4v_{i,j-1} - 4v_{i,j+1} + v_{i,j+2}}{4h} > 0.
\end{align}
\end{subequations}
Recall that the velocity field $\vec{u}$ is incompressible in Model~1, thus one can preprocess the given $\vec{u}$ such that the velocity point values satisfy the following discrete divergence free constraint:
\begin{subequations}\label{eq:2D4order_divfree}
\begin{align}
&\mbox{For odd $i$ and odd $j$,}\nonumber\\
&\frac{-u_{i-2,j} + 4u_{i-1,j} - 4u_{i+1,j} + u_{i+2,j}}{4h}
+ \frac{-v_{i,j-2} + 4v_{i,j-1} - 4v_{i,j+1} + v_{i,j+2}}{4h} = 0;\\
&\mbox{for odd $i$ and even $j$,}\nonumber\\
&\frac{- u_{i-2,j} + 4u_{i-1,j} - 4u_{i+1,j} + u_{i+2,j}}{4h} + \frac{v_{i,j-1} - v_{i,j+1}}{2h} = 0;\\
&\mbox{for even $i$ and odd $j$,}\nonumber\\
&\frac{u_{i-1,j} - u_{i+1,j}}{2h} + \frac{- v_{i,j-2} + 4v_{i,j-1} - 4v_{i,j+1} + v_{i,j+2}}{4h} = 0;\\
&\mbox{for even $i$ and even $j$,}\nonumber\\
&\frac{u_{i-1,j} - u_{i+1,j}}{2h} + \frac{v_{i,j-1} - v_{i,j+1}}{2h} = 0.
\end{align}
\end{subequations}
Then, for any incompressible velocity satisfying the discrete divergence free constraint \eqref{eq:2D4order_divfree}, we know the \eqref{eq:2D4order_suff2} is satisfied for positive measure $\invm$ under the following sufficient condition:
\begin{subequations}\label{eq:2D4order_suff3}
\begin{align}
h\max\{|v_{i,j-1}|,|v_{i,j+1}|\} &\leq 2D\min\{\invm_{i,j-1},\invm_{i,j+1}\} \quad\text{for odd $i$ and even $j$}\\
h\max\{|u_{i-1,j}|,|u_{i+1,j}|\} &\leq 2D\min\{\invm_{i-1,j},\invm_{i+1,j}\} \quad\text{for even $i$ and odd $j$}.
\end{align}
\end{subequations}
Notice that \eqref{eq:2D4order_suff1} implies \eqref{eq:2D4order_suff3}.
Thus, under the condition \eqref{eq:2D4order_suff2} (in particular, under the condition \eqref{eq:2D4order_suff1} for a discrete divergence free velocity field), the matrix $\vecc{A}_d + \vecc{A}^z$ is a nonsingular M-matrix. 
Meanwhile, under the same sufficient condition, we have $\vecc{A}\vec{1}>0$, which indicates $\mathcal{N}^0(\vecc{A}\vec{1})=\emptyset$, namely, the third condition in Theorem~\ref{thm:Lorenz_cond} is trivially satisfied.

Finally, to verify  $\vecc{A}_a^+\leq \vecc{A}^z\vecc{A}_d^{-1}\vecc{A}^s$ in Theorem~\ref{thm:Lorenz_cond}, we only need to compare the outputs of $\mathcal A_a^+(\vec g^{n+1})$
with $\mathcal A^z\circ (\mathcal A_d)^{-1} \circ \mathcal A^s (\vec g^{n+1})$. 
If $x_{ij}$ is a knot, we only need the following inequalities hold:
\begin{itemize}
\setlength{\itemindent}{-1em}
\item For the entry in $\vecc{A}_a^+$ associated with the coefficient of $g_{i-2,j}^{n+1}$ in $\mathcal{A}_a^+(\vec{g}^{\,n+1})_{ij}$.
{\footnotesize
\begin{align}
&\left(-\frac{\Delta{t}}{h}\frac{u_{i-2,j}}{4} + \Delta{t}\,\frac{D(3\invm_{i-2,j} - 4\invm_{i-1,j} + 3\invm_{i,j})}{8h^2}\right)^+ \times \nonumber \\
&\left(\invm_{i-1,j}
+ \Delta{t}\, \frac{D(\invm_{i-2,j} + \invm_{i,j})}{h^2}
+ \Delta{t}\, \frac{D(\invm_{i-1,j-2} + 4\invm_{i-1,j-1} + 18\invm_{i-1,j} + 4\invm_{i-1,j+1} + \invm_{i-1,j+2})}{8h^2}\right) \nonumber \\
\leq&
\left(-\frac{\Delta{t}}{h}u_{i-1,j} + \Delta{t}\,\frac{D(4\invm_{i-2,j} + 12\invm_{i,j})}{8h^2} - \Big(-\frac{\Delta{t}}{h}\frac{u_{i-2,j}}{4} + \Delta{t}\,\frac{D(3\invm_{i-2,j} - 4\invm_{i-1,j} + 3\invm_{i,j})}{8h^2}\Big)^+ \right) \times \nonumber\\
&\left(-\frac{\Delta{t}}{h}\frac{u_{i-2,j}}{2} + \Delta{t}\, \frac{D(3\invm_{i-2,j} + \invm_{i,j})}{4h^2}\right)\label{eq:1D4order:Lorenz2_1}
\end{align}}
\item For the entry in $\vecc{A}_a^+$ associated with the coefficient of $g_{i+2,j}^{n+1}$ in $\mathcal{A}_a^+(\vec{g}^{\,n+1})_{ij}$.
{\footnotesize
\begin{align}
&\left(\frac{\Delta{t}}{h}\frac{u_{i+2,j}}{4} + \Delta{t}\,\frac{D(3\invm_{i,j} - 4\invm_{i+1,j} + 3\invm_{i+2,j})}{8h^2}\right)^+ \times \nonumber\\
&\left(\invm_{i+1,j} + \Delta{t}\, \frac{D(\invm_{i,j} + \invm_{i+2,j})}{h^2} + \Delta{t}\, \frac{D(\invm_{i+1,j-2} + 4\invm_{i+1,j-1} + 18\invm_{i+1,j} + 4\invm_{i+1,j+1} + \invm_{i+1,j+2})}{8h^2}\right) \nonumber\\
\leq& 
\left(\frac{\Delta{t}}{h}u_{i+1,j} + \Delta{t}\,\frac{D(12\invm_{i,j} + 4\invm_{i+2,j})}{8h^2} - \Big(\frac{\Delta{t}}{h}\frac{u_{i+2,j}}{4} + \Delta{t}\,\frac{D(3\invm_{i,j} - 4\invm_{i+1,j} + 3\invm_{i+2,j})}{8h^2}\Big)^+ \right) \times \nonumber\\
&\left(\frac{\Delta{t}}{h}\frac{u_{i+2,j}}{2} + \Delta{t}\, \frac{D(\invm_{i,j} + 3\invm_{i+2,j})}{4h^2}\right) \label{eq:1D4order:Lorenz2_2}
\end{align}}
\item For the entry in $\vecc{A}_a^+$ associated with the coefficient of $g_{i,j-2}^{n+1}$ in $\mathcal{A}_a^+(\vec{g}^{\,n+1})_{ij}$.
{\footnotesize
\begin{align}
&\left(-\frac{\Delta{t}}{h}\frac{v_{i,j-2}}{4} + \Delta{t}\,\frac{D(3\invm_{i,j-2} - 4\invm_{i,j-1} + 3\invm_{i,j})}{8h^2}\right)^+ \times \nonumber\\
&\left(\invm_{i,j-1} + \Delta{t}\, \frac{D(\invm_{i,j-2} + \invm_{i,j})}{h^2} + \Delta{t}\, \frac{D(\invm_{i-2,j-1} + 4\invm_{i-1,j-1} + 18\invm_{i,j-1} + 4\invm_{i+1,j-1} + \invm_{i+2,j-1})}{8h^2}\right) \nonumber\\
\leq& 
\left(-\frac{\Delta{t}}{h}v_{i,j-1} + \Delta{t}\,\frac{D(4\invm_{i,j-2} + 12\invm_{i,j})}{8h^2} - \Big(-\frac{\Delta{t}}{h}\frac{v_{i,j-2}}{4} + \Delta{t}\,\frac{D(3\invm_{i,j-2} - 4\invm_{i,j-1} + 3\invm_{i,j})}{8h^2}\Big)^+ \right) \times \nonumber\\
&\left(-\frac{\Delta{t}}{h}\frac{v_{i,j-2}}{2} + \Delta{t}\, \frac{D(3\invm_{i,j-2} + \invm_{i,j})}{4h^2}\right)\label{eq:1D4order:Lorenz2_3}
\end{align}}
\item For the entry in $\vecc{A}_a^+$ associated with the coefficient of $g_{i,j+2}^{n+1}$ in $\mathcal{A}_a^+(\vec{g}^{\,n+1})_{ij}$.
{\footnotesize
\begin{align}
&\left(\frac{\Delta{t}}{h}\frac{v_{i,j+2}}{4} + \Delta{t}\,\frac{D(3\invm_{i,j} - 4\invm_{i,j+1} + 3\invm_{i,j+2})}{8h^2}\right)^+ \times \nonumber\\
&\left(\invm_{i,j+1} + \Delta{t}\, \frac{D(\invm_{i,j} + \invm_{i,j+2})}{h^2} + \Delta{t}\, \frac{D(\invm_{i-2,j+1} + 4\invm_{i-1,j+1} + 18\invm_{i,j+1} + 4\invm_{i+1,j+1} + \invm_{i+2,j+1})}{8h^2}\right) \nonumber\\
\leq& 
\left(\frac{\Delta{t}}{h}v_{i,j+1} + \Delta{t}\,\frac{D(12\invm_{i,j} + 4\invm_{i,j+2})}{8h^2} - \Big(\frac{\Delta{t}}{h}\frac{v_{i,j+2}}{4} + \Delta{t}\,\frac{D(3\invm_{i,j} - 4\invm_{i,j+1} + 3\invm_{i,j+2})}{8h^2}\Big)^+ \right) \times \nonumber\\
&\left(\frac{\Delta{t}}{h}\frac{v_{i,j+2}}{2} + \Delta{t}\, \frac{D(\invm_{i,j} + 3\invm_{i,j+2})}{4h^2}\right)\label{eq:1D4order:Lorenz2_4}
\end{align}}
\end{itemize}
The above inequalities hold trivially, if the positive part in each inequalities is zero. For seeking a sufficient condition, we only need to consider the case that the positive parts are larger than zero. Let us use \eqref{eq:1D4order:Lorenz2_1} as an example to derive a sufficient condition. The \eqref{eq:1D4order:Lorenz2_2}-\eqref{eq:1D4order:Lorenz2_4} are processed in the same way. Multiply $64(\frac{h^2}{\Delta{t}})^2$ on both side, after some manipulation, we have:
{\small
\begin{align*}
&\Big(8\frac{h^2}{\Delta{t}}\,\invm_{i-1,j}
+ 8D(\invm_{i-2,j} + \invm_{i,j})
+ D(\invm_{i-1,j-2} + 4\invm_{i-1,j-1} + 18\invm_{i-1,j} + 4\invm_{i-1,j+1} + \invm_{i-1,j+2})\Big) \times \\
&\Big(-2hu_{i-2,j} + D(3\invm_{i-2,j} - 4\invm_{i-1,j} + 3\invm_{i,j})\Big)
\leq
\Big(-4h u_{i-2,j} + 2D(3\invm_{i-2,j} + \invm_{i,j})\Big) \times\\
&\Big(2h u_{i-2,j} - 8h u_{i-1,j}
+ D(\invm_{i-2,j} + 4\invm_{i-1,j} + 9\invm_{i,j})\Big).
\end{align*}
}\noindent
Let $E_\Delta=[i-2,i+2]\times[j-2,j+2]$. Denote the largest and smallest values of the invariant measure on $E_\Delta$ by $b = \max_{E_\Delta}\{\invm_{i,j}\}$ and $s = \min_{E_\Delta}\{\invm_{i,j}\}$. Assume the finite difference grid spacing satisfies:
\begin{align}\label{eq:2D4order_suff4}
h\max_{E_\Delta}|\vec{u}_{i,j}| \leq \frac{1}{20} D\min_{E_\Delta}\{\invm_{i,j}\}.
\end{align}
Note that \eqref{eq:2D4order_suff4} implies the condition \eqref{eq:2D4order_suff1}. Then we only need
\begin{align*}
4\Big(11D + 2\frac{h^2}{\Delta{t}}\Big)b
\Big(\frac{1}{10}Ds + D(6b - 4s)\Big)
\leq (8Ds - \frac{1}{5} Ds)(14Ds - \frac{1}{2} Ds).
\end{align*}
Therefore, a sufficient condition is:
\begin{align}\label{eq:2D4order_suff_Lorenz_org}
11D+\frac{h^2}{\Delta{t}} \leq \frac{52.65D s^2}{b(12b-7.8s)}.
\end{align}
If $x_{ij}$ is an edge center (either parallel to $y$-axis or parallel to $x$-axis), we only need the following inequalities hold:
\begin{itemize}
\setlength{\itemindent}{-1em}
\item For the entry in $\vecc{A}_a^+$ associated with the coefficient of $g_{i-2,j}^{n+1}$ in $\mathcal{A}_a^+(\vec{g}^{\,n+1})_{ij}$.
{\footnotesize
\begin{align}
&\left(-\frac{\Delta{t}}{h}\frac{u_{i-2,j}}{4} + \Delta{t}\,\frac{D(3\invm_{i-2,j} - 4\invm_{i-1,j} + 3\invm_{i,j})}{8h^2}\right)^+ \times \nonumber\\
&\left(\invm_{i-1,j} + \Delta{t}\, \frac{D(\invm_{i-2,j}+\invm_{i,j}+\invm_{i-1,j-1}+\invm_{i-1,j+1})}{h^2}\right) \nonumber\\
\leq& 
\left(-\frac{\Delta{t}}{h}u_{i-1,j} + \Delta{t}\,\frac{D(4\invm_{i-2,j} + 12\invm_{i,j})}{8h^2} - \Big(-\frac{\Delta{t}}{h}\frac{u_{i-2,j}}{4} + \Delta{t}\,\frac{D(3\invm_{i-2,j} - 4\invm_{i-1,j} + 3\invm_{i,j})}{8h^2}\Big)^+\right) \times \nonumber\\
& \left(-\frac{\Delta{t}}{h}\frac{u_{i-2,j}}{2} + \Delta{t}\, \frac{D(3\invm_{i-2,j} + \invm_{i,j})}{4h^2}\right)\label{eq:1D4order:Lorenz2_5}
\end{align}}
\item For the entry in $\vecc{A}_a^+$ associated with the coefficient of $g_{i+2,j}^{n+1}$ in $\mathcal{A}_a^+(\vec{g}^{\,n+1})_{ij}$.
{\footnotesize
\begin{align}
&\left(\frac{\Delta{t}}{h}\frac{u_{i+2,j}}{4} + \Delta{t}\,\frac{D(3\invm_{i+2,j} - 4\invm_{i+1,j} + 3\invm_{i,j})}{8h^2}\right)^+ \times \nonumber\\
&\left(\invm_{i+1,j} + \Delta{t}\, \frac{D(\invm_{i,j}+\invm_{i+2,j}+\invm_{i+1,j-1}+\invm_{i+1,j+1})}{h^2}\right) \nonumber\\
\leq& 
\left(\frac{\Delta{t}}{h}u_{i+1,j} + \Delta{t}\,\frac{D(12\invm_{i,j} + 4\invm_{i+2,j})}{8h^2} - \Big(\frac{\Delta{t}}{h}\frac{u_{i+2,j}}{4} + \Delta{t}\,\frac{D(3\invm_{i,j} - 4\invm_{i+1,j} + 3\invm_{i+2,j})}{8h^2}\Big)^+\right) \times \nonumber\\
&\left(\frac{\Delta{t}}{h}\frac{u_{i+2,j}}{2} + \Delta{t}\, \frac{D(\invm_{i,j} + 3\invm_{i+2,j})}{4h^2}\right)\label{eq:1D4order:Lorenz2_6}
\end{align}}
\item For the entry in $\vecc{A}_a^+$ associated with the coefficient of $g_{i,j-2}^{n+1}$ in $\mathcal{A}_a^+(\vec{g}^{\,n+1})_{ij}$.
{\footnotesize
\begin{align}
&\left(-\frac{\Delta{t}}{h}\frac{v_{i,j-2}}{4} + \Delta{t}\,\frac{D(3\invm_{i,j-2} - 4\invm_{i,j-1} + 3\invm_{i,j})}{8h^2}\right)^+ \times \nonumber\\
&\left(\invm_{i,j-1} + \Delta{t}\, \frac{D(\invm_{i-1,j-1}+\invm_{i+1,j-1}+\invm_{i,j-2}+\invm_{i,j})}{h^2}\right) \nonumber\\
\leq& 
\left(-\frac{\Delta{t}}{h}v_{i,j-1} + \Delta{t}\,\frac{D(4\invm_{i,j-2} + 12\invm_{i,j})}{8h^2} - \Big(-\frac{\Delta{t}}{h}\frac{v_{i,j-2}}{4} + \Delta{t}\,\frac{D(3\invm_{i,j-2} - 4\invm_{i,j-1} + 3\invm_{i,j})}{8h^2}\Big)^+\right) \times \nonumber\\
&\left(-\frac{\Delta{t}}{h}\frac{v_{i,j-2}}{2} + \Delta{t}\, \frac{D(3\invm_{i,j-2} + \invm_{i,j})}{4h^2}\right)\label{eq:1D4order:Lorenz2_7}
\end{align}}
\item For the entry in $\vecc{A}_a^+$ associated with the coefficient of $g_{i,j+2}^{n+1}$ in $\mathcal{A}_a^+(\vec{g}^{\,n+1})_{ij}$.
{\footnotesize
\begin{align}
&\left(\frac{\Delta{t}}{h}\frac{v_{i,j+2}}{4} + \Delta{t}\,\frac{D(3\invm_{i,j} - 4\invm_{i,j+1} + 3\invm_{i,j+2})}{8h^2}\right)^+ \times \nonumber\\
&\left(\invm_{i,j+1} + \Delta{t}\, \frac{D(\invm_{i-1,j+1}+\invm_{i+1,j+1}+\invm_{i,j}+\invm_{i,j+2})}{h^2}\right) \nonumber\\
\leq& 
\left(\frac{\Delta{t}}{h}v_{i,j+1} + \Delta{t}\,\frac{D(12\invm_{i,j} + 4\invm_{i,j+2})}{8h^2} - \Big(\frac{\Delta{t}}{h}\frac{v_{i,j+2}}{4} + \Delta{t}\,\frac{D(3\invm_{i,j} - 4\invm_{i,j+1} + 3\invm_{i,j+2})}{8h^2}\Big)^+\right) \times \nonumber\\
&\left(\frac{\Delta{t}}{h}\frac{v_{i,j+2}}{2} + \Delta{t}\, \frac{D(\invm_{i,j} + 3\invm_{i,j+2})}{4h^2}\right)\label{eq:1D4order:Lorenz2_8}
\end{align}}
\end{itemize}
Again, we only need to consider the case that the positive parts in above are larger than zero.
Let us use \eqref{eq:1D4order:Lorenz2_5} as an example to derive a sufficient condition. The \eqref{eq:1D4order:Lorenz2_6}-\eqref{eq:1D4order:Lorenz2_8} are processed in the same way. Multiply $64(\frac{h^2}{\Delta{t}})^2$ on both side, after some manipulation, we have:
{\small
\begin{align*}
&\Big(4\frac{h^2}{\Delta{t}}\invm_{i-1,j} + 4D(\invm_{i-2,j}+\invm_{i,j}+\invm_{i-1,j-1}+\invm_{i-1,j+1})\Big) \times\\
&\Big(-2h u_{i-2,j} + D(3\invm_{i-2,j} - 4\invm_{i-1,j} + 3\invm_{i,j})\Big)
\leq \Big(-2h u_{i-2,j} + D(3\invm_{i-2,j} + \invm_{i,j})\Big)\times \\
&\Big(2h u_{i-2,j} - 8hu_{i-1,j} + D(\invm_{i-2,j} + 4\invm_{i-1,j} + 9\invm_{i,j})\Big). 
\end{align*}}\noindent
Recall that $E_\Delta=[i-2,i+2]\times[j-2,j+2]$ and $b = \max_{E_\Delta}\{\invm_{i,j}\}$ and $s = \min_{E_\Delta}\{\invm_{i,j}\}$. Assume the finite difference grid spacing satisfies \eqref{eq:2D4order_suff4}. Then  we only need
\begin{align*}
4\Big(8D + 2\frac{h^2}{\Delta{t}}\Big)b
\Big(\frac{1}{10}Ds + D(6b - 4s)\Big)
\leq (8Ds - \frac{1}{5} Ds)(14Ds - \frac{1}{2} Ds).
\end{align*}
Therefore, \eqref{eq:2D4order_suff_Lorenz_org} still serves as a sufficient condition.
%
Now, let us try to simplify above sufficient condition.
The invariant measure $\invm \geq \epsilon_0 > 0$, define $r=b/s$, then \eqref{eq:2D4order_suff_Lorenz_org} can be rewritten as
\begin{align*}
\frac{h^2}{\Delta{t}} 
\leq \frac{52.65D}{r(12r-7.8)} - 11D
= 6.75 D\Big(\frac{1}{r-0.65} - \frac{1}{r}\Big) - 11D.
\end{align*}
From the definition of $r$, we know $r \geq 1$. Thus, it is sufficient to employ the conditions $r\in[1,1.025]$ and
\begin{align*}
\frac{h^2}{\Delta{t}} \leq \sqrt{2}D
< \min_{r\in[1,1.025]} \left\{6.75 D\Big(\frac{1}{r-0.65} - \frac{1}{r}\Big) - 11D\right\}.
\end{align*}
This indicates we only need to find a suitable upper bound on $h$ such that $b \leq 1.025s$ (namely $r\in[1,1.025]$) holds.
Recall $\invm$ is continuously differentiable. 
Assume $\invm$ take its maximum at point $\vec{x}^\ast$ on $E_\Delta$ and $\invm$ take its minimum at point $\vec{x}_\ast$ on $E_\Delta$. By mean value theorem, there exist a point $\vec{\xi} \in E_\Delta$ such that
\begin{align*}
\invm(\vec{x}^\ast) = \invm(\vec{x}_\ast) + (\vec{x}^\ast - \vec{x}_\ast)\cdot\nabla{\invm{(\vec{\xi})}}.
\end{align*}
Therefore
\begin{align*}
b \leq s + 4\sqrt{2}\,h \max_{E_\Delta}|(\invm_x',\invm_y')|,
\end{align*}
which means in order to let $b \leq 1.025s$ hold, we can employ a sufficient condition as follows
\begin{align}
s + 4\sqrt{2}\,h \max_{E_\Delta}|(\invm_x',\invm_y')| \leq 1.025s.
\end{align}
To this end, we obtain a constraint on $h$, as follows
\begin{align}\label{eq:2D4order_suff5}
h \max_{E_\Delta}|(\invm_x',\invm_y')| \leq \frac{\sqrt{2}}{320}\min_{E_\Delta}\{\invm_{i,j}\}.
\end{align}
As a summary, we have the following theorem:
\begin{thm}\label{thm:2D4order_suff_conds}
Under the mesh and time step constraints \eqref{eq:2D4order_suff2}, \eqref{eq:2D4order_suff4}, \eqref{eq:2D4order_suff5} and $\frac{\Delta t}{h^2}\geq \frac{1}{\sqrt2 D}$, the coefficient matrix for the unknown vector $\vec{g}$ in the fourth order finite difference scheme \eqref{2d-4th-scheme} satisfies the Lorenz's conditions, so it is a product of two M-matrices thus  monotone.
In particular, let $E_\Delta=[i-2,i+2]\times[j-2,j+2]$, with a two dimensional discrete divergence free velocity field satisfying \eqref{eq:2D4order_divfree},  then the matrix in fourth order finite difference scheme \eqref{2d-4th-scheme} is monotone under the following constraints:
\begin{align*}
&h\max_{E_\Delta}|\vec{u}_{i,j}| \leq \frac{1}{20} D\min_{E_\Delta}\{\invm_{i,j}\},\\
&h \max_{E_\Delta}|(\invm_x',\invm_y')| \leq \frac{\sqrt{2}}{320}\min_{E_\Delta}\{\invm_{i,j}\},\\
&\frac{\Delta t}{h^2}\geq \frac{1}{\sqrt2 D}. 
\end{align*}
\end{thm}
\begin{rem}
 We emphasize that the conditions above are only convenient sufficient conditions for monotonicity, rather than sharp necessary conditions. However, the monotonicity in the fourth order finite difference scheme \eqref{2d-4th-scheme} will be lost in numerical tests if $\Delta t$ approaches $0$. So certain lower bound on $\frac{\Delta t}{h^2}$ is a necessary condition for monotonicity. 
\end{rem}

\section{Properties of the  fully discrete numerical schemes}\label{sec:positivity_and_dissipation}
We only discuss the two dimensional case since all the results can be easily reduced to the one dimensional case. 
The discussion in this section  holds for both the second order scheme \eqref{2d-2nd-scheme}  and the fourth order scheme \eqref{2d-4th-scheme}. 
For convenience, we use $\vec{g}_h$  to denote the numerical solution vector in two dimensions with entries $g_{i}$ ($i=1,\cdots, N^2$).
The finite element space $V^h$ is $N^2$-dimensional with Lagrangian $Q^k$ basis $\{\phi_i(\mathbf x)\}_{i=1}^{N^2}$ defined at the $(k+1)\times(k+1)$-point Gauss-Lobatto points.

\subsection{Natural properties of the finite element method}
Since the finite difference schemes in Section \ref{sec:numerical_scheme} are derivied from a finite element method, they inherit many good properties from
the original finite element method, which will be used for deriving energy dissipation. 
We can express the numerical scheme \eqref{eq:Model1_quadform}  in a matrix-vector form. We introduce the following matrices:
\begin{align*}
[\vecc{W}]_{i,j} &= \langle\phi_j, \phi_i\rangle, &
[\vecc{M}]_{i,j} &= \mathrm{diag}(\invm_{1}, \cdots, \invm_{N^2}), \\
[\vecc{A}^\mathrm{diff}]_{i,j} &= \langle D\invm \nabla{\phi_j}, \nabla{\phi_i}\rangle, &
[\vecc{A}^\mathrm{adv}]_{i,j} &= \langle u \phi_j, \nabla{\phi_i}\rangle.
\end{align*}
Since we use the Gauss-Lobatto quadrature, the lumped mass matrix $\vecc{W}$ is a diagonal matrix, with quadrature weights $\omega_i>0$ on the diagonal. 
Then the matrix-vector form of \eqref{eq:Model1_quadform} is
\begin{align}\label{eq:matrix_form_scheme}
(\vecc{W}\vecc{M} + \Delta{t}\vecc{A}^\mathrm{diff} + \Delta{t}\vecc{A}^\mathrm{adv})\vec{g}^{\,n+1}_h 
= \vecc{W}\vecc{M}\vec{g}^{\,n}_h.
\end{align} 
For simplicity, we define 
$$\vecc{A}:=\vecc{W}\vecc{M} + \Delta{t}\vecc{A}^\mathrm{diff} + \Delta{t}\vecc{A}^\mathrm{adv},$$
$$ \vecc{B}:=\vecc{M}^{-1} \vecc{W}^{-1} \vecc{A}=\vecc{I}+ \Delta{t}\vecc{M}^{-1} \vecc{W}^{-1}\vecc{A}^\mathrm{diff} + \Delta{t}\vecc{M}^{-1} \vecc{W}^{-1}\vecc{A}^\mathrm{adv}.$$
Thus $\vec{g}^{\,n+1}_h= \vecc{A}^{-1}\vecc{W}\vecc{M}\vec{g}^{\,n}_h=\vecc{B}^{-1}\vec{g}^{\,n}_h$.

Consider an arbitrary test function $\phi_h\in V^h$ with point values $\phi_i$ $(i=1,\cdots N^2)$. Let $\vec\phi_h$ be the vector with entries $\phi_i$. Then the scheme  \eqref{eq:Model1_quadform} is equivalent to the following matrix-vector form:
\[ \vec{\phi}_h^{\,T} \vecc{A} \vec{g}^{n+1}_h=\vec{\phi}_h^{\,T} \vecc{W}\vecc{M} \vec{g}^{n}_h,\quad \forall \vec{\phi}_h\in \mathbb  R^{N^2}.\]

By considering the test function $\phi_h\equiv 1$, we get 
$$\forall g_h\in V^h,\quad \langle u g_h , \nabla 1 \rangle=0\Longrightarrow \vec 1^{\,T} \vecc{A}^\mathrm{adv}\vec {g}_h =0,  \forall\vec {g}_h\in \mathbb  R^{N^2} \,\Longrightarrow  \vec 1^{\,T} \vecc{A}^\mathrm{adv}=\vec 0, $$
$$\forall  g_h\in V^h,\quad \langle D\mathcal M \nabla g_h , \nabla 1 \rangle=0\Longrightarrow \vec 1^{\,T} \vecc{A}^\mathrm{diff}\vec {g}_h =0,  \forall\vec {g}_h\in \mathbb  R^{N^2} \,\Longrightarrow  \vec 1^{\,T} \vecc{A}^\mathrm{diff}=\vec 0. $$

Thus we have
\begin{equation}
 \label{scheme-property} \vec 1^{\,T} \vecc{A}=\vec 1^{\,T} (\vecc{W}\vecc{M} + \Delta{t}\vecc{A}^\mathrm{diff} + \Delta{t}\vecc{A}^\mathrm{adv})=\vec 1^{\,T} \vecc{W}\vecc{M}.  \end{equation}

 The next natural property of the finite element method \eqref{eq:Model1_quadform} is $\vecc{M}^{-1}\vecc{W}^{-1}\vecc{A}\vec{1}=\vec{1}$, under the assumption that the velocity field satisfies the following discrete divergence free constraint:
 \begin{equation}
 \forall \phi_h\in V^h,\quad \langle u , \nabla \phi_h \rangle=0.
 \label{general-div-free}
 \end{equation} 
 It is straightforward to verify that \eqref{general-div-free} is equivalent to \eqref{2d-disccrete_div_free} in the second order scheme, and equivalent to \eqref{eq:2D4order_divfree} in the fourth order scheme. 
 
 Notice that we first have
 \[ g_h\equiv 1 \Longrightarrow \langle D\mathcal M \nabla g_h, \nabla \phi_h\rangle=0,  \forall \phi_h\in V^h  \Longrightarrow  \vec\phi_h^{\, T}\vecc{A}^\mathrm{diff} \vec 1=0,  \forall\vec {\phi}_h\in \mathbb  R^{N^2}\Longrightarrow  \vecc{A}^\mathrm{diff} \vec 1=\vec 0. \]
 With the discrete divergence free condition \eqref{general-div-free}, we have
 \[ g_h\equiv 1 \Longrightarrow \langle u  g_h, \nabla \phi_h\rangle=0,  \forall \phi_h\in V^h  \Longrightarrow  \vec\phi_h^{\, T}\vecc{A}^\mathrm{adv} \vec 1=0,  \forall\vec {\phi}_h\in \mathbb  R^{N^2}\Longrightarrow  \vecc{A}^\mathrm{adv} \vec 1=\vec 0. \]
 Therefore
\begin{equation}\label{groundstate}
\vecc{M}^{-1}\vecc{W}^{-1}\vecc{A}\vec{1} = (\vecc{I}+ \Delta{t}\vecc{M}^{-1} \vecc{W}^{-1}\vecc{A}^\mathrm{diff} + \Delta{t}\vecc{M}^{-1} \vecc{W}^{-1}\vecc{A}^\mathrm{adv})\vec{1} 
=\vec{1}.
\end{equation}

\subsection{Mass conservation}

By plugging in the test function $\phi_h\equiv 1$ in \eqref{eq:Model1_quadform}, we get 
$\langle \mathcal M g_h^{n+1}, 1\rangle=\langle \mathcal M g_h^{n}, 1\rangle$, thus
$$\langle  \rho_h^{n+1}, 1\rangle=\langle \mathbf \rho_h^{n}, 1\rangle,$$
which can also be written as
\begin{align}\label{eq:discrete_mass_conv}
\sum_{i=1}^{N^2}\omega_i \rho_i^{n+1} = \sum_{i=1}^{N^2}\omega_i \rho_i^{n}.
\end{align} 

\subsection{Steady state preserving}
 
If $\vec{g}_h^{\,n} = K\vec{1}$ for some constant $K$, then multiply $\vecc{M}^{-1}\vecc{W}^{-1}$ on both side of \eqref{eq:matrix_form_scheme}, we have 
\begin{equation}\label{scheme_f}
\vecc{M}^{-1}\vecc{W}^{-1}\vecc{A}\vec{g}_h^{\,n+1}=\vec{g}_h^{\,n}. 
\end{equation}
It is a well known fact that  the stiffness matrix $\vecc{A}$ in the finite element method \eqref{eq:Model1_quadform} is nonsingular, which is implied by the coercivity of the bilinear form in \eqref{eq:Model1_varform} for an incompressible velocity field. 
When the linear system above is nonsingular, it is straightforward to verify that the  unique solution is $\vec{g}_h^{\,n+1} = K\vec{1}$. Therefore, in terms of density, we have $\rho^n_i = K\invm_i,\,\forall i$ implies $\rho^{n+1}_i = K\invm_i,\,\forall i$.

\subsection{Positivity} 
 
At time step $n$, assume $\rho^n_i > 0$ for every $i$, then $g^n_i = \rho^n_i / \invm_i \geq 0$ for every $i$, since invariant measure is positive.  
If all suitable mesh and time step constraints hold so that all the monotonicity results in Section \ref{sec:monotonicity} hold,  
then $\vecc{A}^{-1} \geq 0$ holds. Since $\mathcal M_i>0$ and $\omega_i>0$, we have 
$$\vec{g}^{\,n+1}_h= \vecc{A}^{-1}\vecc{W}\vecc{M}\vec{g}^{\,n}_h\geq 0$$
thus $\rho^{n+1}_i = \invm_i g^{n+1}_i 
\geq 0$. 

\subsection{Energy dissipation}

For any convex function $f(x)$, define the discrete energy at time step $n$ as
\begin{align*}
E(\rho^n_h) = \langle \invm f(\frac{\rho^n_h}{\invm}), 1 \rangle
= \sum_{i=1}^{N^2} \omega_i \invm_i f\left (\frac{\rho^{n}_i}{\invm_i}\right )= \langle \invm f(g_h^n), 1 \rangle= \sum_{i=1}^{N^2} \omega_i \invm_i f(g^n_i).
\end{align*}

\begin{thm}
Assume the velocity field satisfies the discrete divergence free constraint \eqref{general-div-free}, e.g,  \eqref{2d-disccrete_div_free} in the second order scheme, and \eqref{eq:2D4order_divfree} in the fourth order scheme. 
If the scheme \eqref{eq:Model1_quadform} is monotone, then for any convec function $f(x)$, it dissipates the discrete energy:
\[\sum_{i=1}^{N^2} \omega_i \invm_i f\left (\frac{\rho^{n+1}_i}{\invm_i}\right )\leq\sum_{i=1}^{N^2} \omega_i \invm_i f\left (\frac{\rho^{n}_i}{\invm_i}\right ). \]
\end{thm}
\begin{proof}
 Let $a^{ij}$ be the entries of  $\vecc{A}^{-1}$. Then $\vec{g}^{\,n+1}_h= \vecc{A}^{-1}\vecc{W}\vecc{M}\vec{g}^{\,n}_h$ gives
\begin{equation}
g^{n+1}_i=\sum\limits_{j} a^{ij}\omega_j \mathcal M_j g^n_j. 
\label{proof-convexcomb}
\end{equation}
Next we show that \eqref{proof-convexcomb} is a convex combination due to monotonicity and natural properties of the finite element method. 
The monotonicity implies $a^{ij}\geq 0$. The   property $\vecc{M}^{-1}\vecc{W}^{-1}\vecc{A}\vec 1=\vec 1$ gives 
$$\vec 1  =\vecc{A} ^{-1} \vecc{W}\vecc{M} \vec 1 \Longrightarrow  \sum\limits_{j} a^{ij}\omega_j \mathcal M_j=1. $$
Thus \eqref{proof-convexcomb} is a convex combination. 
For a convex function $f$,  Jensen's inequality gives
\[ f(g^{n+1}_i)\leq \sum\limits_{j} a^{ij}\omega_j \mathcal M_j f(g^n_j). \]

On the other hand, the property \eqref{scheme-property} implies  
  \[ \vec 1^{\,T} \vecc{A}=\vec 1^{\,T} \vecc{W}\vecc{M} \Longrightarrow\vec 1^{\,T}=\vec 1^{\,T} \vecc{W}\vecc{M} \vecc{A}^{-1}  \Longrightarrow \sum\limits_{i} a^{ij}\omega_i \mathcal M_i=1 \]
So we have
\begin{align}\label{eq:discrete_engy_law}
E^{n+1} &= \sum_{i}\omega_i \invm_i f(\frac{\rho_i^{n+1}}{\invm_i}) =
\sum_{i}\omega_i \invm_i f(g_i^{n+1}) \\
&\leq \sum_{i}\omega_i \invm_i \left [\sum\limits_{j} a^{ij}\omega_j \mathcal M_j f(g^n_j)\right] 
= \sum_{j}\omega_j \invm_j \left [\sum\limits_{i} a^{ij}\omega_i \mathcal M_i \right] f(g^n_j)
= \sum_{j}\omega_j \invm_j f(g^n_j) = E^{n}.\nonumber
\end{align}

\end{proof}

\begin{rem}
As a special case, by choosing the convex function $f(x)=(x-1)^2$ and using the discrete mass conservation \eqref{eq:discrete_mass_conv}, the discrete energy dissipation law \eqref{eq:discrete_engy_law} reduces to the following form
\begin{align}
\sum_{i}\omega_i \left(\frac{\rho_i^{n+1}}{\invm_i}-1 \right)^2 \invm_i
< \sum_{i}\omega_i \left(\frac{\rho_i^{n}}{\invm_i}-1 \right)^2 \invm_i,
\end{align}
which is viewed as a discrete energy dissipation law w.r.t the Pearson $\chi^2$-divergence.
\end{rem}

\section{Numerical Tests}
\label{sec-test}
\subsection{Accuracy test}

We consider the scheme \eqref{eq:Model1_quadform}
solving 
\begin{equation}\label{test_f}
 \rho_t=\nabla \cdot \bbs{ \invm \nabla \frac{\rho}{\invm}} + \vec{u} \cdot \nabla \frac{\rho}{\invm} + f 
\end{equation}
on $\Omega=(0,\pi)\times(0,\pi)$ with no flux boundary condition, i.e., $\nabla \frac{\rho}{\mathcal M}\cdot \vec n=0.$
We test the second order and fourth order spatial accuracy on a steady state solution 
\begin{equation}\label{exact_sol}
\rho(x,y,t)=(3\cos x \cos y+3)(2+\sin x\sin y),
\end{equation}
with $\mathcal M=2+\sin x\sin y$,  $\vec u=\langle \sin x\cos y, \cos x\sin y\rangle $. The source $f(x,y)$ is chosen such that $\rho(x,y,t)$ above is the exact solution to \eqref{test_f}. 

The time step is set as $\Delta t=\Delta x$ and errors at $T=1$ are given 
in Table \ref{table-accuracy} where $l^2$ error is defined as 
$$\sqrt{\Delta x \Delta y\sum_i \sum_j |u_{ij}-u(x_i, y_j)|^2}$$ with $u_{ij}$ and $u(x,y)$ denoting the numerical and exact solutions, respectively. 
We observe the expected order of spatial accuracy.

\begin{table} [htbp]
    \resizebox{\textwidth}{!}{%
\begin{tabular}{|c|c|c|c|c|c|c|c|c|}
\hline
\multirow{2}{*}{FD Grid} & 
\multicolumn{4}{c|}{the second order scheme  \eqref{2d-2nd-scheme}}  &
\multicolumn{4}{c|}{the fourth order scheme  \eqref{2d-4th-scheme}} \\ 
\cline{2-9}
 & 
 $l^2$ error & order & $l^\infty$ error & order  &
$l^2$ error & order & $l^\infty$ error & order  \\
\hline 
$9\times 9$  & 
    2.99E-1     & --- &     2.93E-1     & --- & 
    1.66E-2     & --- &     1.17E-2     & --- 
 \\
\hline
$17\times 17$ & 
    6.00E-2     &     2.32     &     8.38E-2     &     1.81     & 
    9.98E-4     &     4.05     &     8.15E-4     &     3.84    
 \\
\hline
$33\times 33$ &
    1.21E-2     &     2.31     &     2.21E-2     &     1.92     & 
    6.14E-5     &     4.02    &     5.31E-5     &     3.94    
 \\
\hline
$65\times 65$ &
    2.59E-3     &     2.23     &     5.67E-3     &     1.96    & 
    3.81E-6     &     4.01     &     3.31E-6     &     4.00     
 \\ \hline
 $129\times 129$ &
    5.85E-4     &     2.15     &     1.44E-3     &     1.98     & 
    2.37E-7     &     4.01     &     2.07E-7     &     4.00     
 \\ \hline
\end{tabular}}
\label{table-accuracy}
\caption{Accuracy test for a steady state solution \eqref{exact_sol} to the Fokker-Planck equation \eqref{test_f}  with a source term.}
\end{table}

\subsection{Numerical examples with a given sampling target $\mathcal M$}
 
We consider examples with a given sampling target $\pi$.
On a 2D domain $\Omega=[-4.5,4.5]\times[-4.5,4.5]$,
 the   stream function for a 2D sinusoidal cellular flow is given as
\begin{equation}\label{psi}
\psi(x,y):= A\sin  (k\pi x) \sin  (k\pi y),
\end{equation} 
where $A$ represents the amplitude of the mixture velocity $\vec{u}$ and $k$ is the normalized wave number  of the mixture.
Then the incompressible velocity field is given as
$
\vec{u} = \bbs{\begin{array}{cc}
-\pt_y \psi\\
\pt_x \psi
\end{array}
} $.

The target density is taken to be  a smiling triple-banana image:
\begin{align}\label{smile}
\mathcal M(x,y) &= e^{ -20\big[\bbs{x-\frac65}^2+\bbs{y-\frac{6}{5}}^2 - \frac12\big]^2  + \log\bbs{e^{-10(y-2)^2}} } +  e^{-20 \big[\bbs{x+\frac65}^2+\bbs{y-\frac{6}{5}}^2 - \frac12\big]^2 + \log\bbs{e^{-10(y-2)^2}}} \nonumber\\
 &+ e^{-20 \bbs{x^2+y^2 - 2}^2 + \log\bbs{e^{-10(y+1)^2}}}+0.1.
\end{align}
Then we take a  Gaussian mixture as the initial density
\begin{equation}\label{smile-initial}
\rho^0(x,y)= e^{-16(x+3)^2-4y^2}+ e^{-16(x-3)^2-4y^2}+ e^{-4x^2-16(y+3)^2}+ e^{-4x^2-16(y-3)^2}+0.1.
\end{equation} 

The numerical solutions for both second order and fourth order schemes, as well as the energy evolution for $E=\int_\Omega \frac{\rho^2}{\mathcal M} d\mathbf x$,  are given in Figure \ref{FP-figure}. From the color contour, no visual difference can be observed.  The positivity-preserving and energy decay can be proved for both schemes.

      \begin{figure}[!htbp]
      \begin{center}
 \subfigure{\includegraphics[scale=0.35]{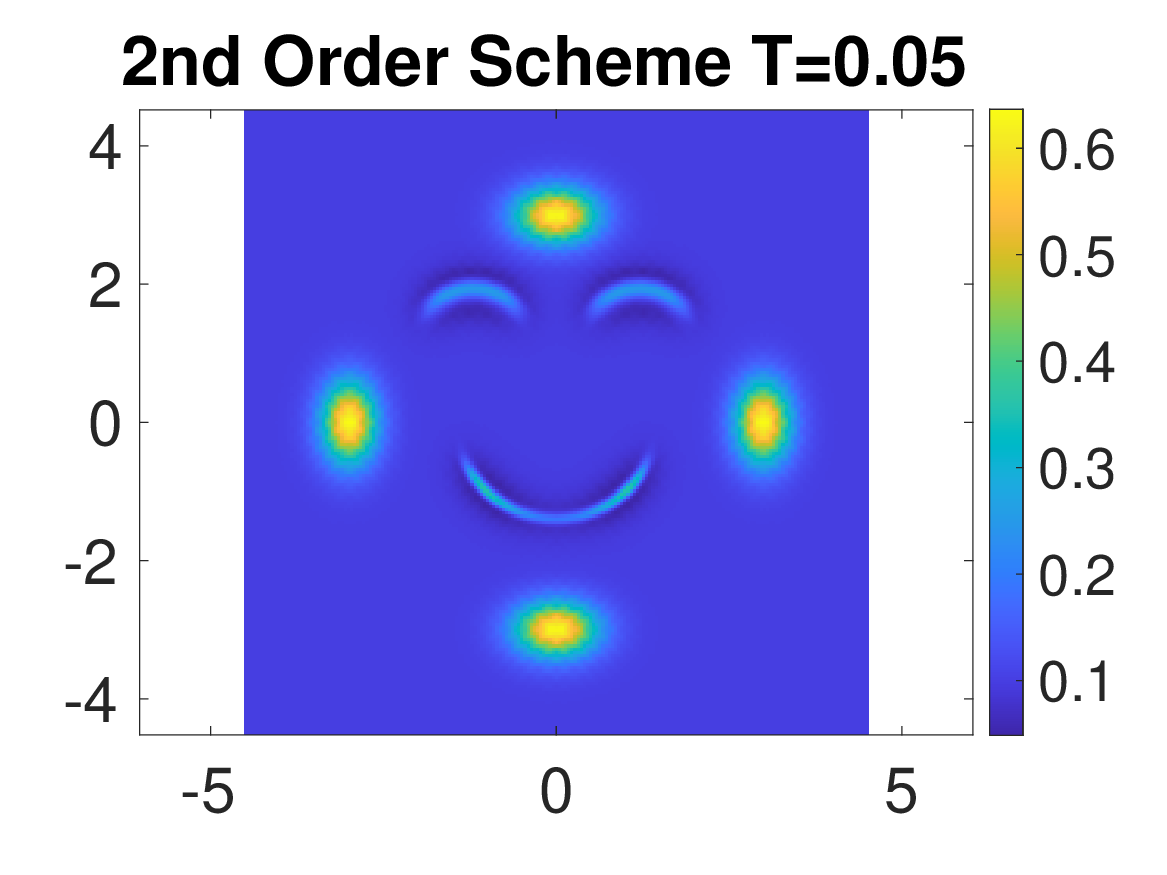} }
 \hspace{-.1in}
 \subfigure{\includegraphics[scale=0.35]{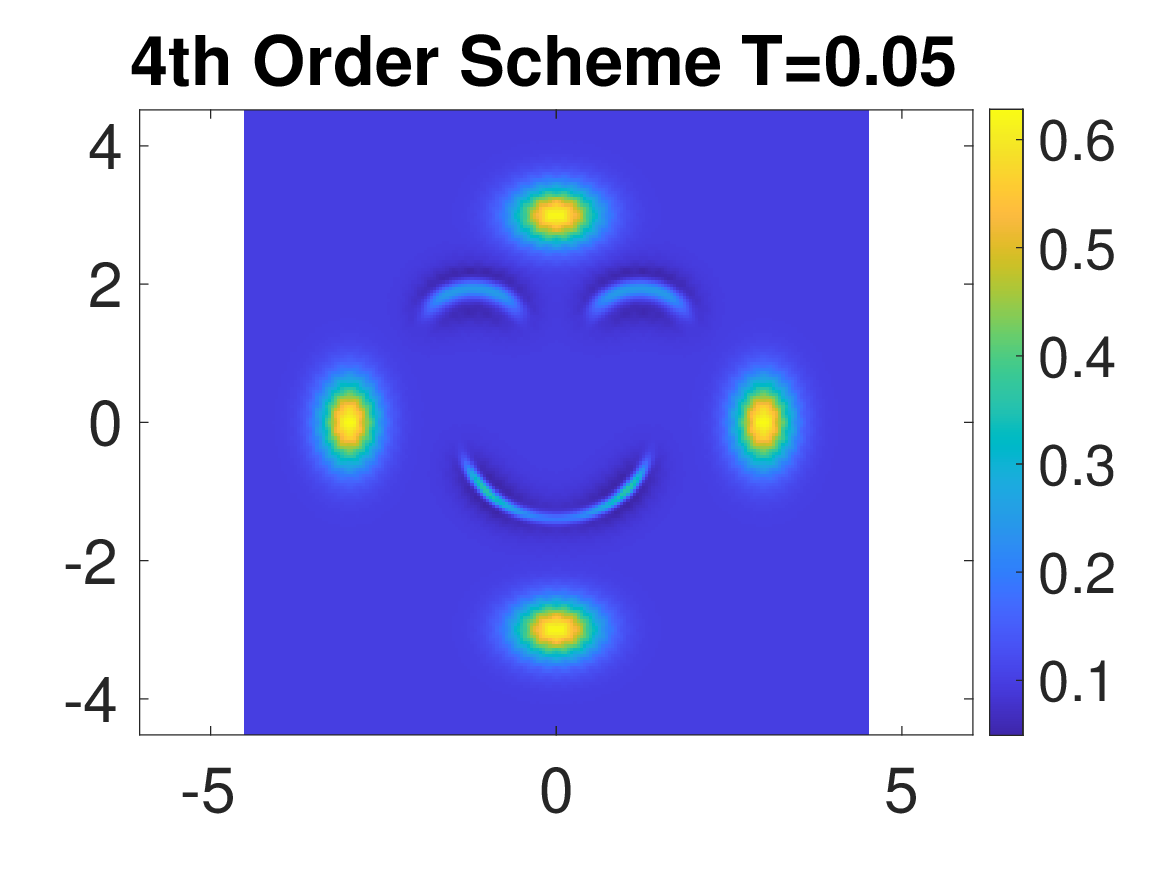}}\\\vspace{-0.1in}
 \subfigure{\includegraphics[scale=0.35]{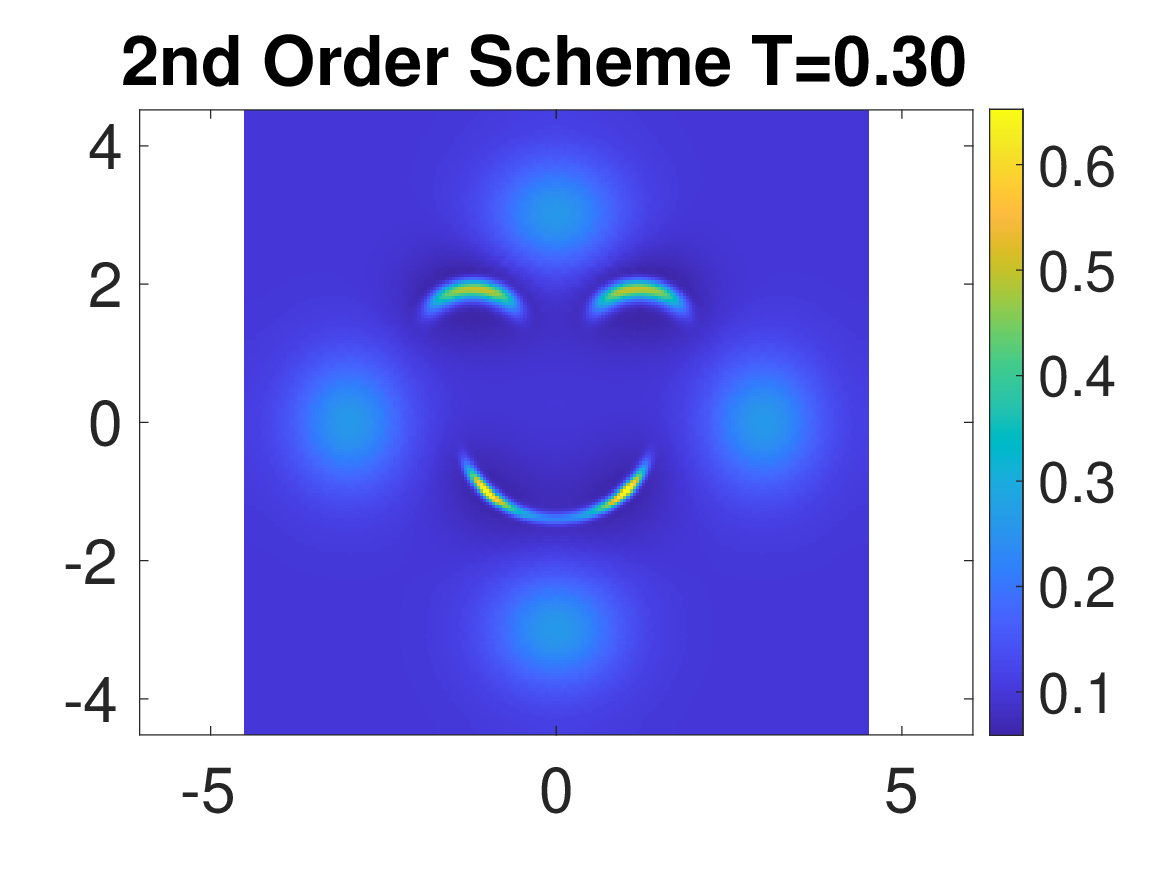} }
 \hspace{-.1in}
 \subfigure{\includegraphics[scale=0.35]{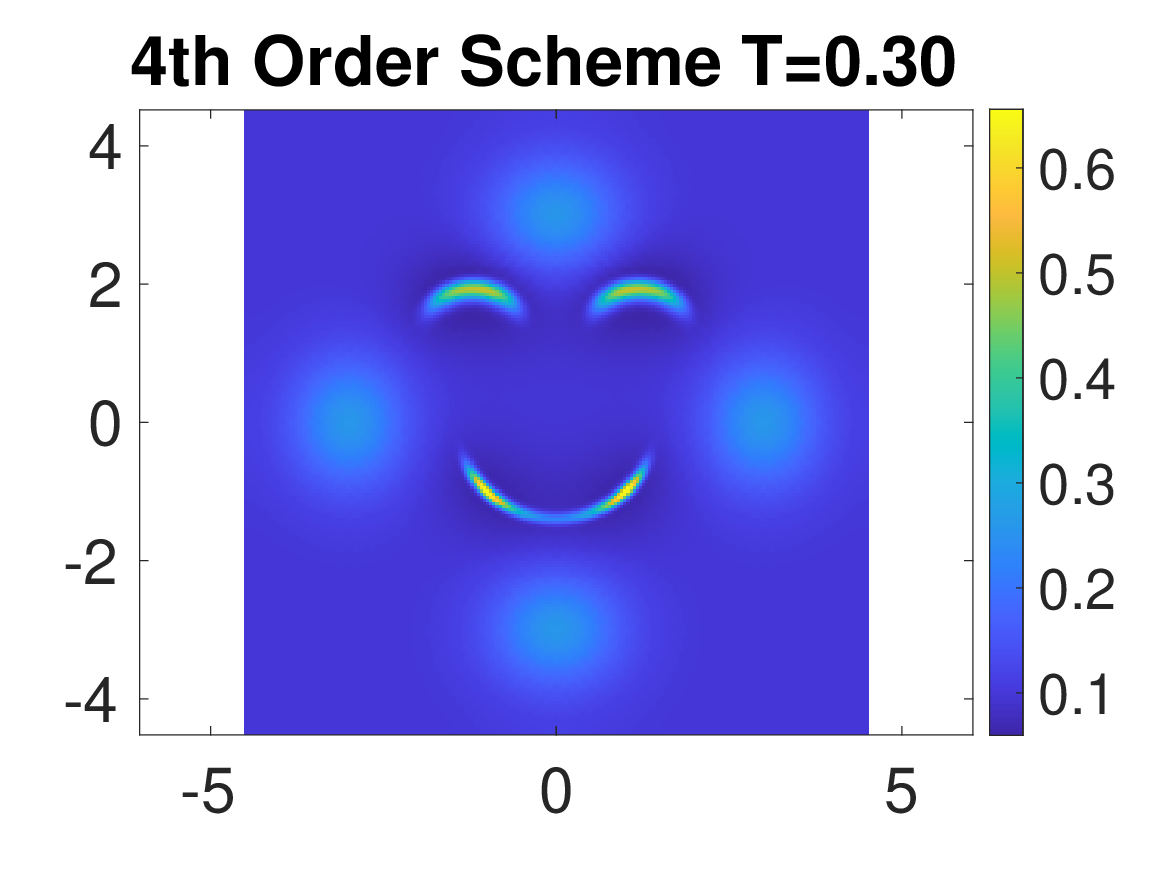}}\\\vspace{-0.1in}
 \subfigure{\includegraphics[scale=0.35]{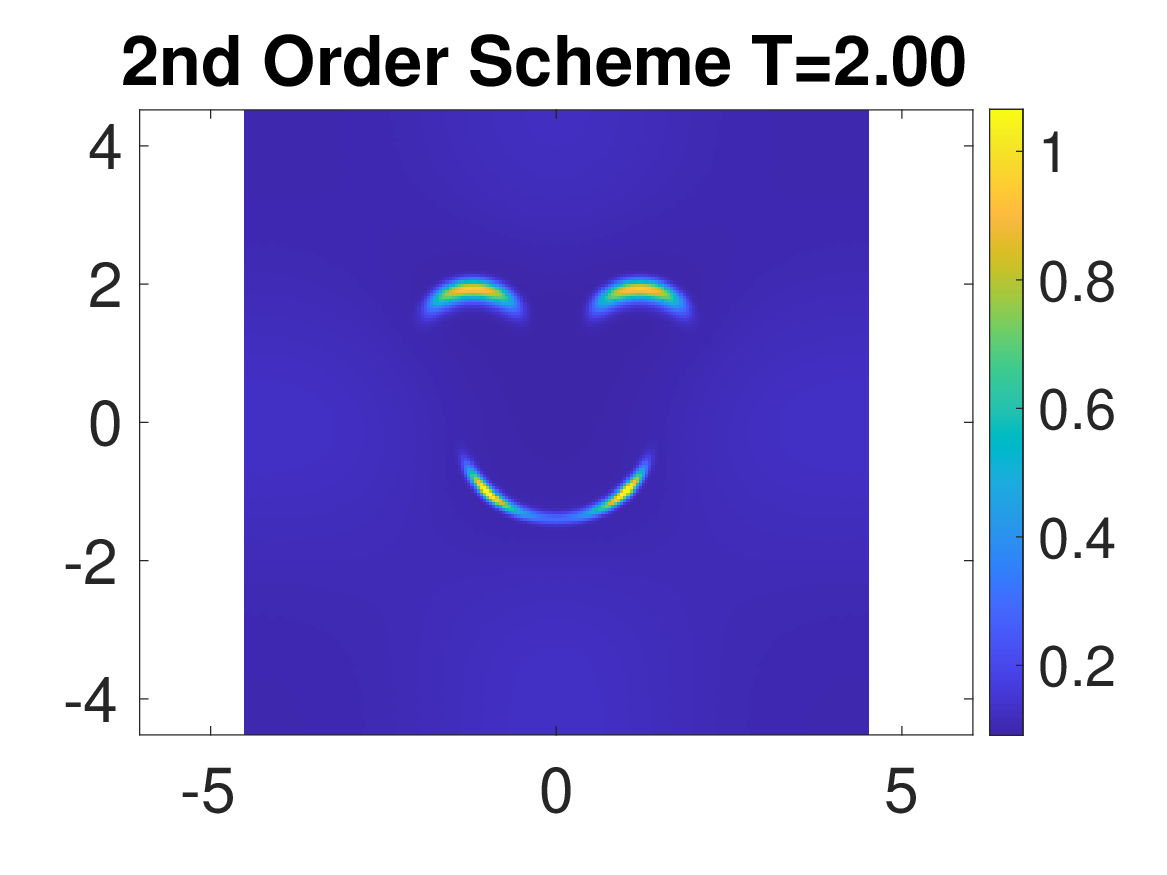} }
 \hspace{-.1in}
 \subfigure{\includegraphics[scale=0.35]{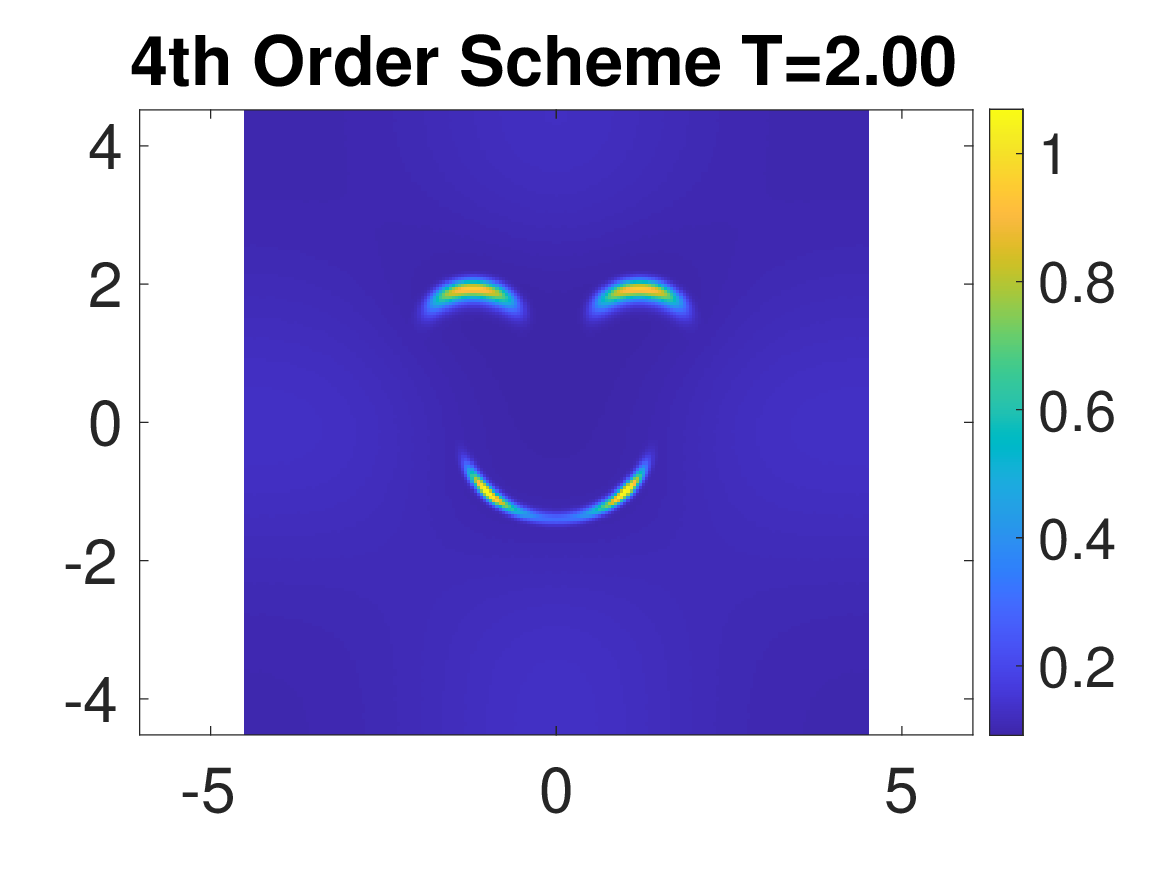}}\\
 \vspace{-0.1in}
  \subfigure{\includegraphics[scale=0.35]{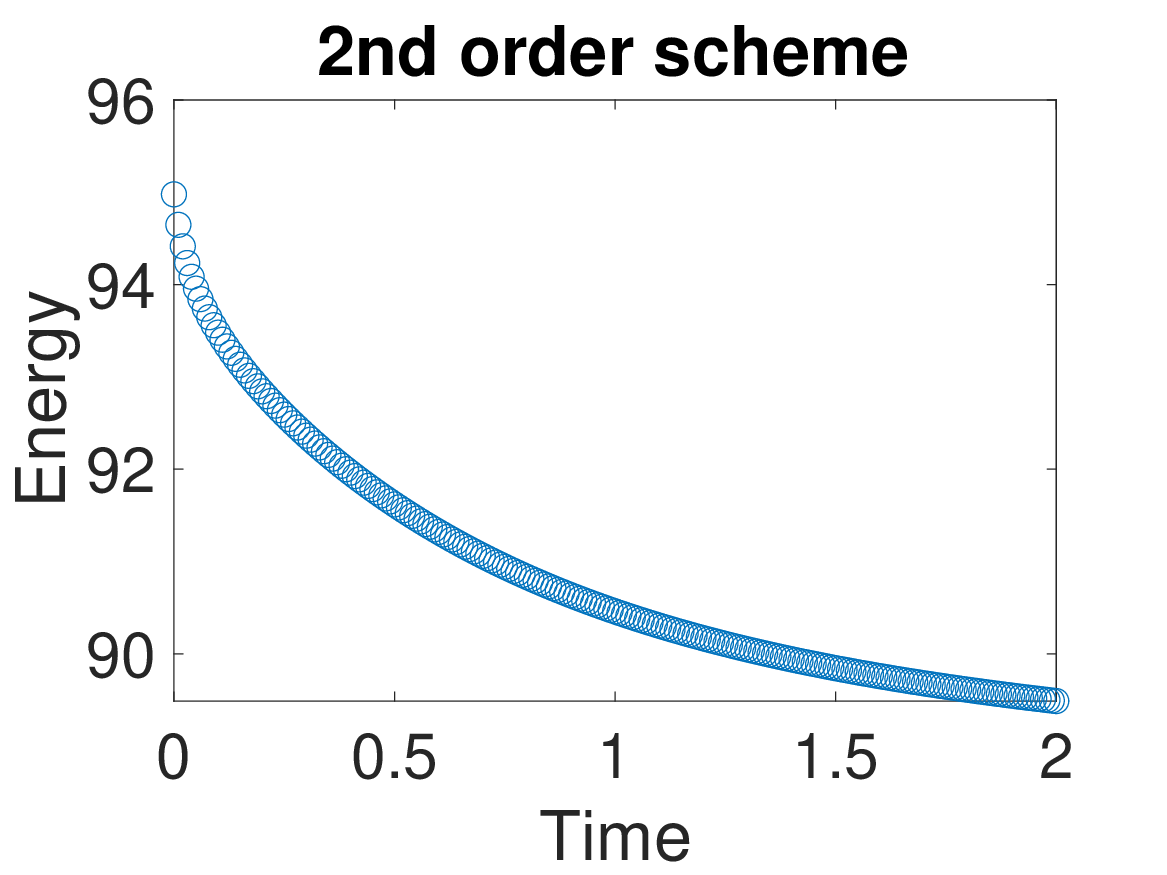} }
 \hspace{-.1in}
 \subfigure{\includegraphics[scale=0.35]{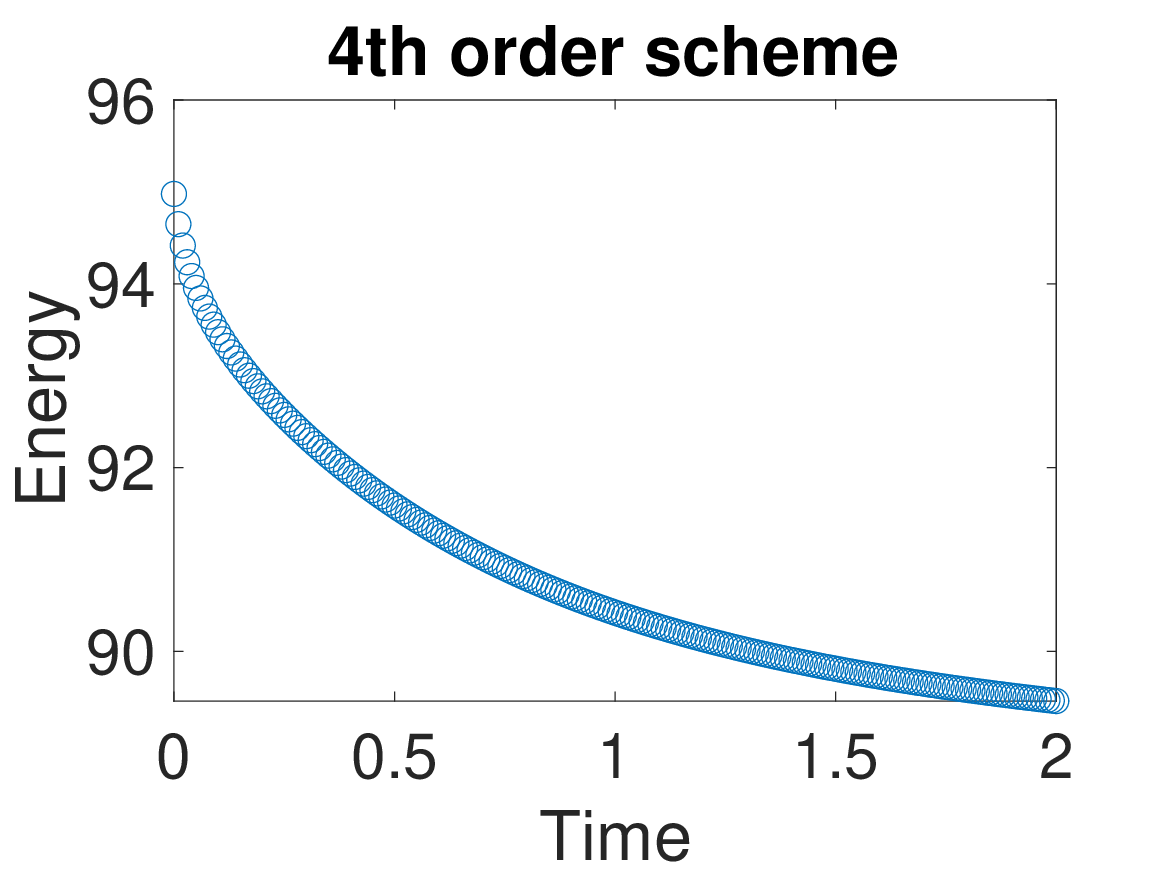}}
 \caption{Numerical solutions to the Fokker-Planck equation with a smiling triple-banana target density \eqref{smile} and the initial density \eqref{smile-initial}. Both the second order scheme \eqref{2d-2nd-scheme} and the fourth order scheme \eqref{2d-4th-scheme} are used on the same $201\times 201$ grid with $\Delta t=0.01$.  }
\label{FP-figure}
\end{center}
 \end{figure}

Next we consider a different example, in which the fourth order spatial discretization can produce visually better results than the second order one.  
The computational domain is $[-3,3]\times[-3, 3]$, the diffusion constant is $D=0.5$, and the velocity filed is defined by derivatives of the stream function \eqref{psi} with $A=0.2$ and $k=1$. 

Now the initial data is chosen to be
\begin{equation}\label{initial-2}
\rho^0(x,y)= \frac12 e^{-16(x+1)^2-4y^2}+ \frac12 e^{-16(x-1)^2-4y^2}+ e^{-x^2/4-(y+3)^2}+ e^{-x^2/4-(y-3)^2}+0.1,
\end{equation}
while the target density is
\begin{equation}\label{target-2}
\mathcal M (x,y)= e^{-(x+3)^2-y^2/4}+ e^{-(x-3)^2-y^2/4}+ \frac12 e^{-4x^2-16(y+1)^2}+ \frac12 e^{-4x^2-16(y-1)^2}+0.1.
\end{equation}

See Figure \ref{FP-figure2} for the numerical solutions of the second example. The second order scheme on the finest mesh $301\times 301$ grid with time step $\Delta t=0.005$ can be regarded as the reference solution. On the same coarse $101\times 101$ grid with the same time step $\Delta t=0.02$, we observe that the second order scheme produces a wrong solution, while the fourth order scheme produces a better solution. A preconditioned conjugate gradient method is used to solve the linear systems in the semi-implicit schemes, and the cost for both second order and fourth order schemes on the same grid is about the same. Thus the fourth order scheme has clear advantages, even though the time discretization is only first order.

       \begin{figure}[!htbp]
      \begin{center}
 \subfigure[The second order scheme on a $101\times 101$ grid with $\Delta t=0.02$.]{\includegraphics[scale=0.32]{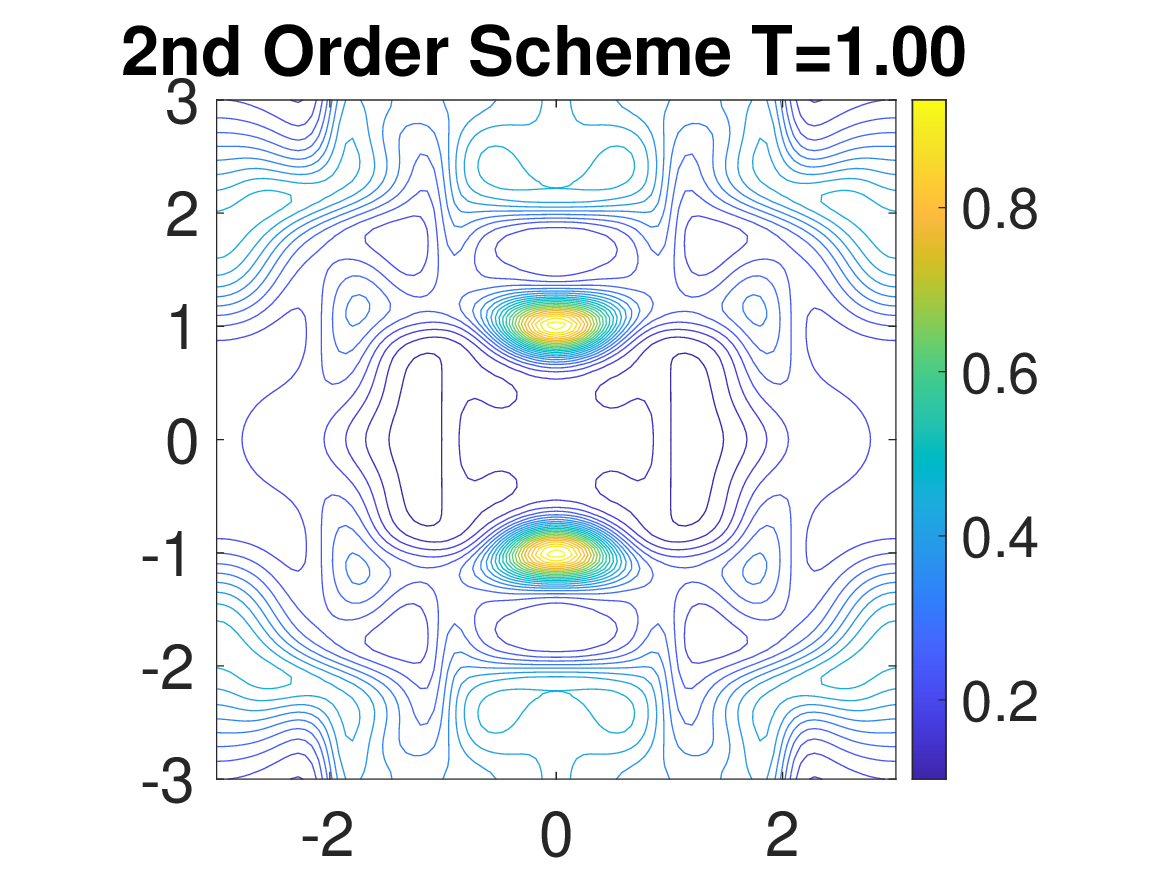} }
  \subfigure[The fourth order scheme on a $101\times 101$ grid with $\Delta t=0.02$.] {\includegraphics[scale=0.32]{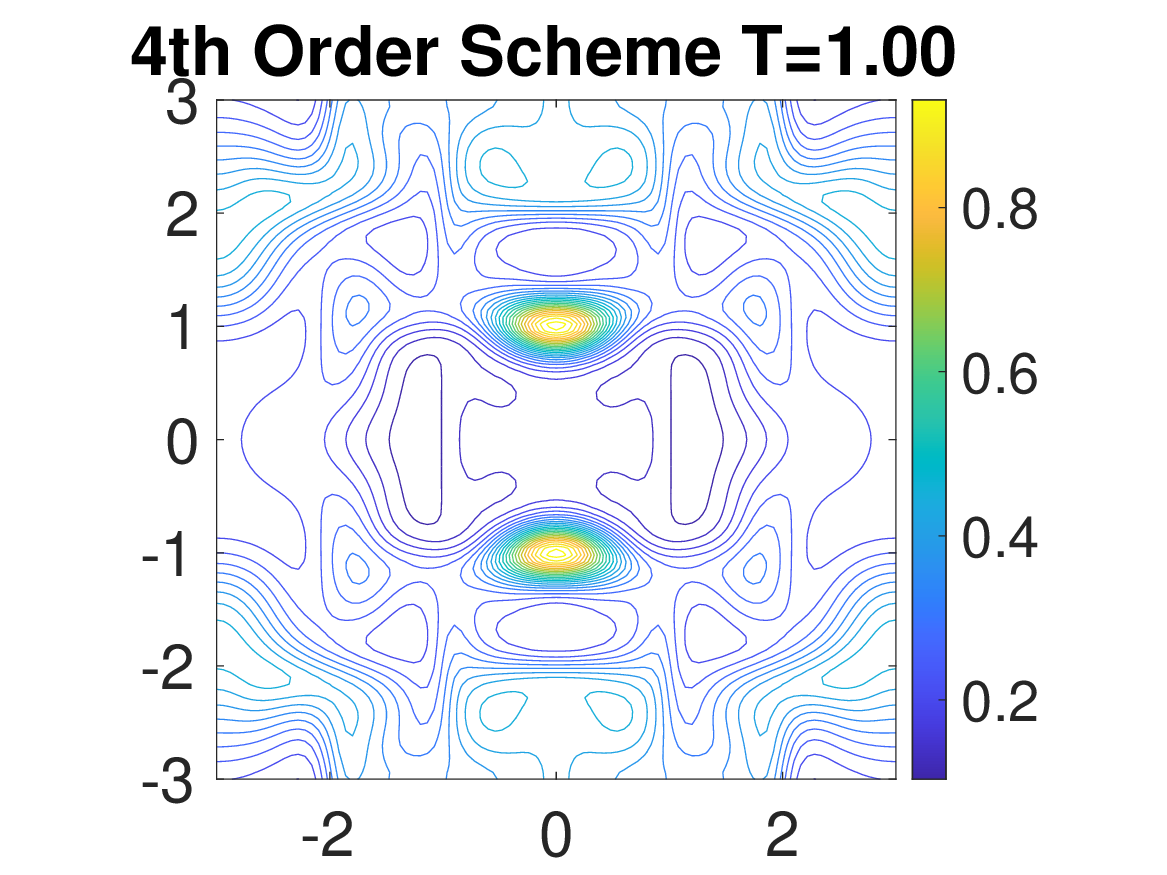}}\\\vspace{-0.1in}
   \subfigure[Zoomed-in of Figure (a).]{\includegraphics[scale=0.32]{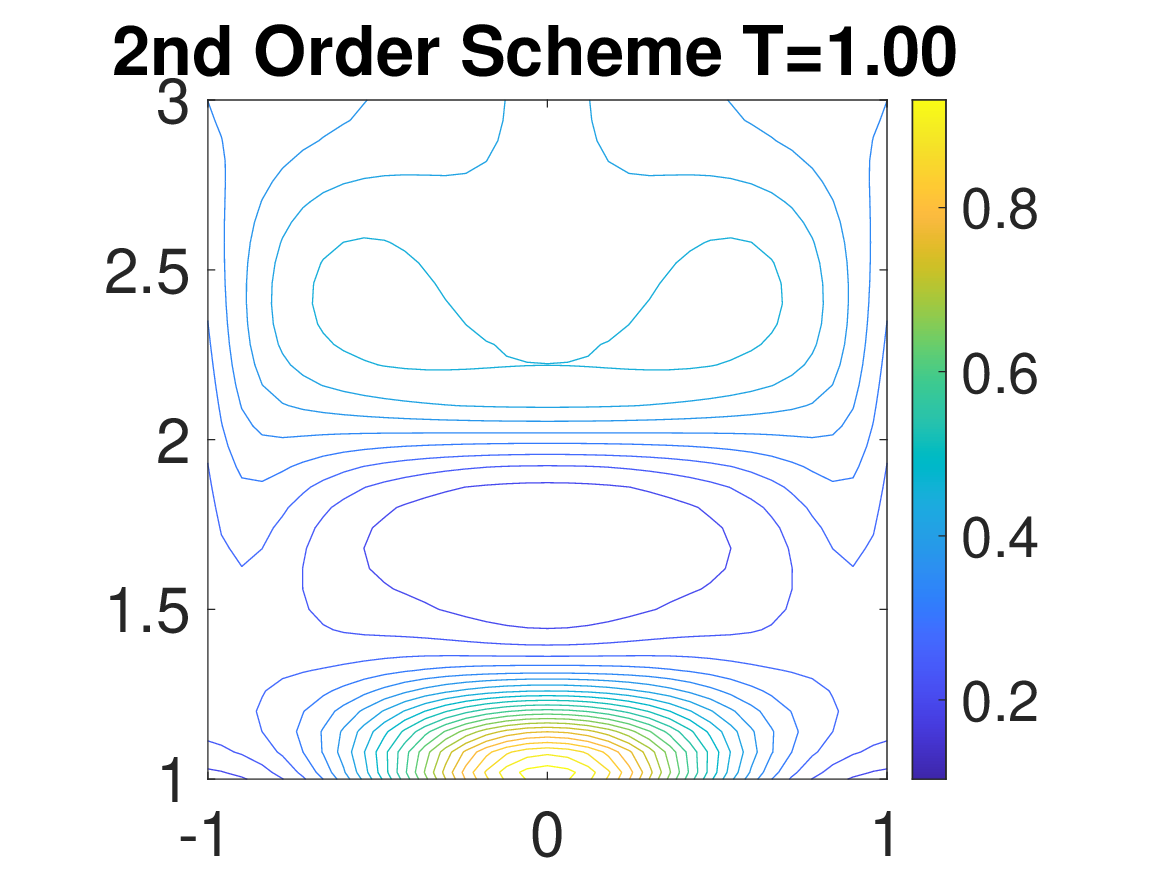} }
    \subfigure[Zoomed-in of Figure (b).]{\includegraphics[scale=0.32]{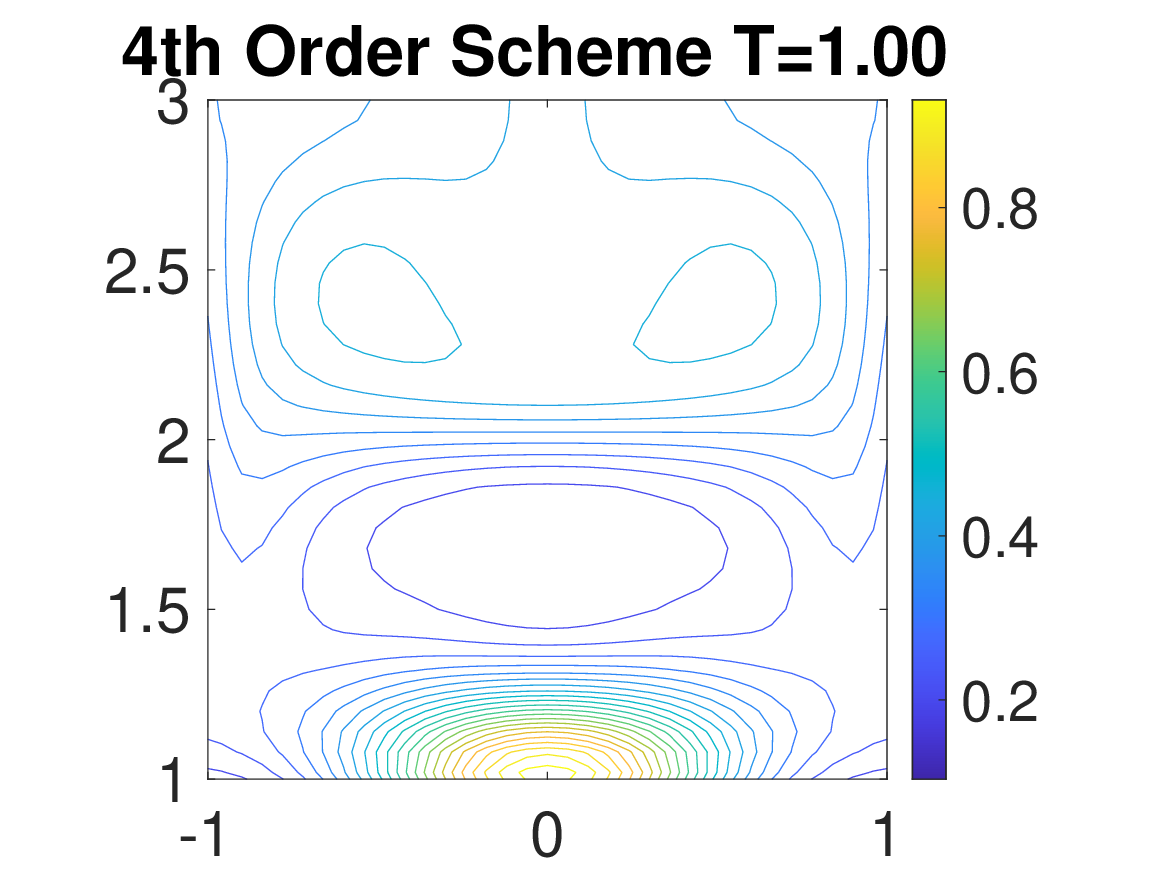} }\\\vspace{-0.1in}
 \subfigure[The second order scheme on a $201\times 201$ grid with $\Delta t=0.01$.] {\includegraphics[scale=0.32]{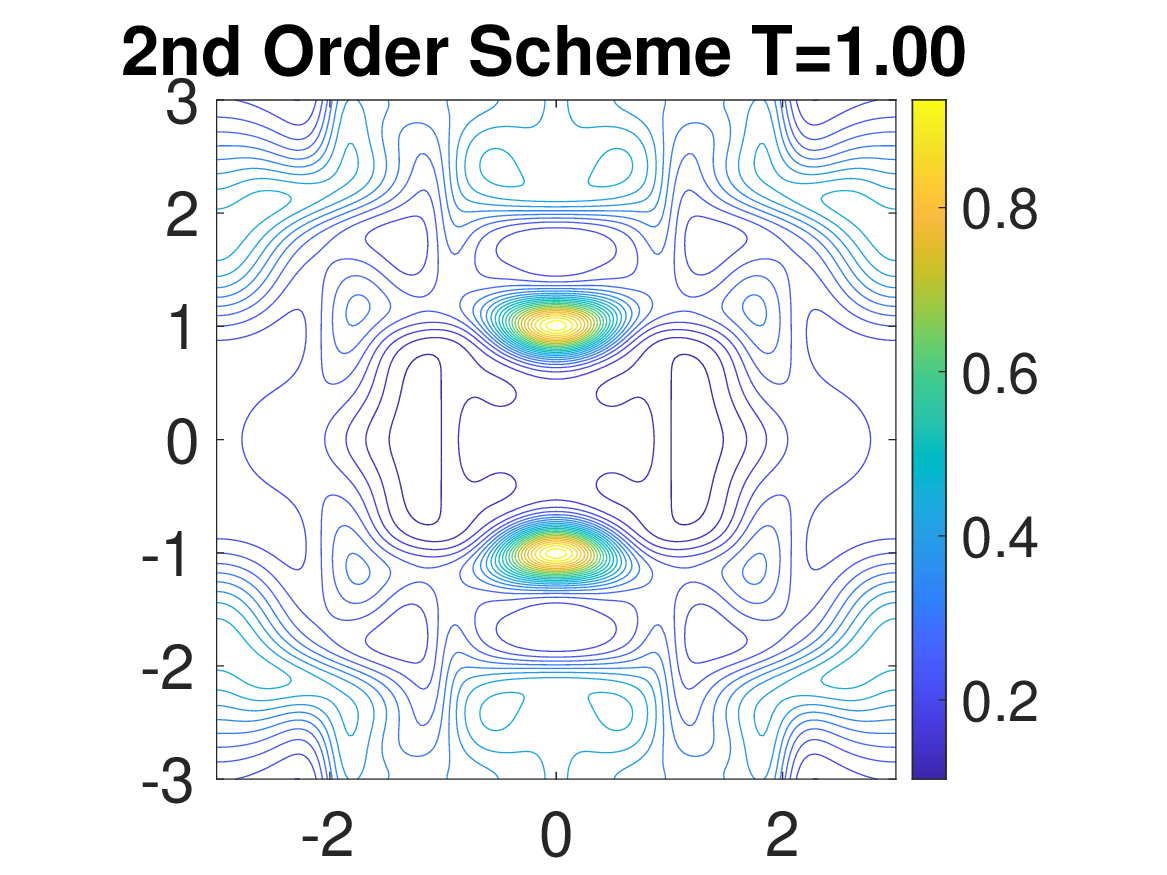} }
 \subfigure[The second order scheme on a $301\times 301$ grid with $\Delta t=0.005$.]{\includegraphics[scale=0.32]{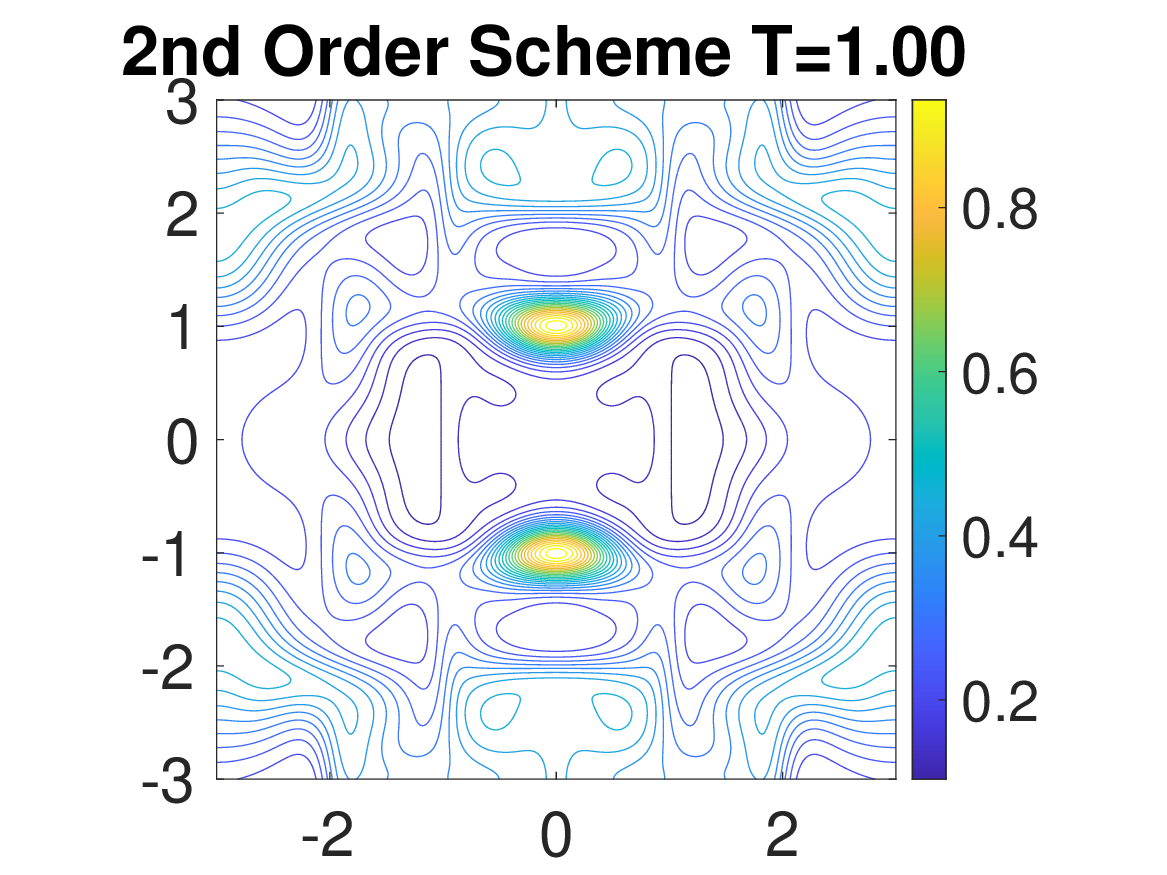}}\\\vspace{-0.1in}
 
 \caption{Numerical solutions to the Fokker-Planck equation with the target density \eqref{target-2} and the initial density \eqref{initial-2}.
 The solution on the finest grid in Figure (f) can be regarded as the reference solution. By comparing Figures (a), (b), (c) and (d) with the reference solution in (f), 
 we can observe that the fourth order scheme   \eqref{2d-4th-scheme} produces better results than the second order scheme  \eqref{2d-2nd-scheme}   on a coarse $101\times 101$ grid.   }
\label{FP-figure2}
\end{center}
 \end{figure}

\section{Concluding remarks}
\label{sec-remark}
In this paper, we have constructed second order and fourth order space discretization via finite difference implementation of the finite element method for solving   Fokker-Planck equations associated with  irreversible processes. Under mild mesh conditions and time step constraints for smooth solutions, the high order schemes are proved to be monotone, thus are positivity-preserving and energy dissipative. Even though the time discretization is only first order, numerical tests suggest that the fourth order spatial scheme produces better solutions than the second order one on the same grid. The high order schemes proposed in this paper preserve all the good properties just as the classical first order upwind schemes, such as (\emph{i}) the conservation of total mass, (\emph{ii}) the positivity of $\rho$, (\emph{iii}) the energy dissipation law with respect to $\phi$-entropy, and (\emph{iv}) exponential convergence to equilibrium $\invm$. Those properties are  important but difficult to obtain for high order space discretizations, particularly for irreversible drift-diffusion processes. This also enables the future studies for sampling acceleration and variance reduction using  irreversible processes with high order numerical schemes.

\section*{Data Availibility} All data generated or analysed during this study are included in this  article.

\newpage
\printbibliography
\end{document}